\documentclass[mathpazo]{cicp}

\renewcommand{\d}{\mathrm{d}}
\newcommand{\T}{\mathrm{T}}
\newcommand{\avg}[1]{\langle{#1}\rangle}
\newcommand{\Ge}{\geqslant}
\newcommand{\Le}{\leqslant}
\newcommand{\Prec}{\preccurlyeq}
\newcommand{\Succ}{\succcurlyeq}
\newcommand{\Ent}{\mathrm{\,Ent}}
\newcommand{\std}{\mathrm{std}}

\usepackage{physics,bm,booktabs,graphicx,url,xcolor,caption}
\usepackage[page]{appendix}
\newcommand{\red}[1]{#1}

\theoremstyle{theorem}

\theoremstyle{definition}
\newtheorem*{assumption}{Assumption}

\begin{document}
\title[Ergodicity of PIMD]{\red{Dimension-free Ergodicity} of Path Integral Molecular Dynamics}


\author{
Xuda Ye\affil{1}\comma\corrauth,
Zhennan Zhou\affil{2}
}

\address{
\affilnum{1}Beijing International Center for Mathematical Research, Peking University, Beijing 100871, China. \\
\affilnum{2}Institute for Theoretical Sciences, Westlake University, Hangzhou, 310030, China.
}

\emails{ \\
{\tt abneryepku@pku.edu.cn} (X.~Ye)\\
{\tt zhouzhennan@westlake.edu.cn} (Z.~Zhou)
}

\begin{abstract}
The quantum thermal average plays a central role in describing the thermodynamic properties of a quantum system. Path integral molecular dynamics (PIMD) is a prevailing approach for computing quantum thermal averages by approximating the quantum partition function as a classical isomorphism on an augmented space, enabling efficient classical sampling, but the theoretical knowledge of the \red{ergodicity of the sampling} is lacking. \red{Parallel to the standard PIMD with $N$ ring polymer beads, we also study the Matsubara mode PIMD, where the ring polymer is replaced} by a continuous loop composed of $N$ \red{Matsubara modes}. \red{Utilizing} the generalized $\Gamma$ calculus, we prove that \red{both the Matsubara mode PIMD and the standard PIMD} have uniform-in-$N$ ergodicity, i.e., \red{the convergence rate towards the invariant distribution does not depend on the number of modes or beads $N$.}
\end{abstract}

\ams{37A30, 82B31, 81S40}
\keywords{quantum thermal average, path integral molecular dynamics,
Matsubara modes, ergodicity, generalized $\Gamma$ calculus.}

\maketitle

\section{Introduction}
Calculation of the quantum thermal average plays an important role in quantum physics and quantum chemistry, not only because it fully characterizes the canonical ensemble of the quantum system, but also because of its wide applications in describing the thermal properties of complex quantum systems, including the idea quantum gases \cite{q1}, chemical reaction rates \cite{q2,q3}, density of states of crystals \cite{q4}, quantum phase transitions \cite{q5}, etc. However, since the computational cost of direct discretization methods (finite difference, pseudospectral methods, etc.) grows exponentially with the spatial dimension \cite{s1,s2}, the exact calculation of the quantum thermal average is hardly affordable in high dimensions. In the past decades, there have been numerous methods committed to compute the quantum thermal average approximately, and the path integral molecular dynamics (PIMD) is among the most prevailing ones.

The PIMD is a computational framework to obtain accurate quantum thermal averages. Using the imaginary time-slicing approach in Feynman path integral \cite{feynman},
the PIMD maps the quantum system in $\mathbb R^d$ to a ring polymer of $N$ beads in $\mathbb R^{dN}$, where each bead represents a classical duplicate of the original quantum system, and adjacent beads are connected by a harmonic spring potential. When the number of beads $N$ is large enough, the classical Boltzmann distribution of the ring polymer in $\mathbb R^{dN}$ is expected to yield an accurate approximation of the quantum thermal average.
Since the development of the PIMD in 1970s, this framework has been widely used in the calculations in the chemical reaction rates
\cite{pimd_6,pimd_9,pimd_10}, transition state theory \cite{pimd_7} and tunneling splittings \cite{ts_1,ts_2}.
\red{Several} variants of the PIMD, the ring polymer molecular dynamics (RPMD) \cite{rpmd_1,rpmd_2}, the centroid molecular dynamics (CMD) \cite{cmd_1,cmd_2} and \red{the Matsubara dynamics \cite{Matsubara_0,Matsubara_1}}, have been employed to compute the quantum correlation function.
Recently, there have been fruitful studies on designing efficient and accurate algorithms to enhance the numerical performance of the PIMD \cite{pimd_3,pimd_4,pimd_5,pimd_8}.
The PIMD can also be applied in multi-electronic-state quantum systems, see \cite{multi_1,multi_2} for instance.
The general procedure to compute the quantum thermal average in the PIMD is demonstrated as follows:
\begin{enumerate}
\item \textbf{Ring polymer approximation}. Pick a positive integer $N$ and map the quantum system to a ring polymer system of $N$ beads, where each bead is a classical duplicate of the original system.
\item \textbf{Construction of sampling}. \red{Equip the ring polymer system with a stochastic thermostat (e.g., Langevin, Andersen and Nos\'e--Hoover) so that the resulting stochastic process samples} the Boltzmann distribution of the ring polymer.
\item \textbf{Stochastic simulation}. Evolve the stochastic process with a proper time discretization method, and approximate the quantum thermal average by the time average in the long-time simulation.
\end{enumerate}
\red{The PIMD framework based on the ring polymer beads is referred to as the standard PIMD in this paper}.

Although the standard PIMD has become a mature framework to compute the quantum thermal average, \red{a rigorous} mathematical understanding \red{of its convergence} is still sorely lacking. A fundamental question is the \red{dimension-free} ergodicity, i.e., does the stochastic process sampling the Boltzmann distribution has a convergence rate which does not depend on the number of beads $N$?
\red{The property of} uniform-in-$N$ ergodicity has been partially validated when the potential function is quadratic \cite{pmmLang,Cayley_2}, but a rigorous justification in the general case is still to be investigated.

\red{An important goal of this paper is to justify the dimension-free ergodicity of the standard PIMD. Nevertheless, our proof mainly relies on an alternative PIMD framework, the Matsubara mode PIMD, which is more convenient for analysis.
As its name indicates, this framework is based on the Matsubara modes \cite{Matsubara,Matsubara_0} rather than the discrete beads of the ring polymer.} \red{In the Matsubara mode PIMD, the ring polymer is expressed as the superpositions of $N$ Matsubara modes, resulting in a continuous imaginary-time ring polymer loop rather than a ragged one (see Figure~3.2 of \cite{Matsubara_0} for example)}.
\red{The Matsubara mode PIMD has a close relation to the standard PIMD, because
the Matsubara mode PIMD with suitable discretization is equivalent to the standard PIMD \cite{Matsubara,Matsubara_compare}, except for the minor difference in the mode frequencies. In fact, the standard PIMD utilizes the normal mode frequencies \cite{Matsubara_0}, whose continuum limit give rise to the Matsubara frequencies. It is also pointed in \cite{Matsubara,Matsubara_compare} that the standard PIMD has better performance than the Matsubara mode PIMD in computing the quantum thermal average.}

The main contribution of this paper is that we prove the uniform-in-$N$ ergodicity of the \red{Matsubara mode PIMD and the standard PIMD}, namely, the convergence rate towards the invariant distribution does not depend on $N$, \red{which stands for the number of modes in the Matsubara mode PIMD and the number of beads in the standard PIMD. The authors believe that this is the first dimension-free egordicity result of the PIMD in the form of the underdamped Langevin dynamics. Our proof strategies include} the Bakry--\'Emery calculus \cite{bakry} and the \red{recently developed} generalized $\Gamma$ calculus \cite{m1,m2}.
\red{It should be stressed that the convergence rate can still depend on the potential function $V(x)$ and the temperature, and in particular the convergence rate is exponentially small in the low temperature limit. This a a common feature of the Bakry--\'Emery calculus with a non-convex potential function, referring to Proposition~5.1.6 of \cite{bakry}.}

The paper is organized as follows. In Section \ref{section: derivation} we derive the \red{Matsubara mode} PIMD from the \red{Feynman path integral representation of the quantum thermal average}. In Section \ref{section: convergence analysis} we prove the uniform-in-$N$ ergodicity the \red{Matsubara mode PIMD and the standard PIMD}. In Section \ref{section: numerical result} we implement the numerical tests to show the performance of the \red{Matsubara mode PIMD and the standard PIMD} in computing the quantum thermal average and validate the dimension-free ergodicity.
\section{Derivation of the \red{Matsubara mode} PIMD}
\label{section: derivation}
In this section we derive the \red{Matsubara mode PIMD from the Feynman path integral representation of the quantum thermal average}. First, we show that the quantum equilibrium can be equivalently interpreted as the classical Boltzmann distribution of a continuous loop representing the ring polymer in the position space. Second, we represent the energy function $\mathcal E(\xi)$ of the ring polymer loop in terms of the \red{Matsubara} coordinates $\xi$. Finally, we construct an infinite-dimensional underdamped Langevin dynamics to sample the distribution $\exp{-\mathcal E(\xi)}$, formulating the \red{Matsubara mode} PIMD.

\red{Note that the our derivation does not rely on the normal mode transform, and is thus different from the traditional ways to generate the Matsubara modes \cite{Matsubara_0}.}
\subsection{Quantum equilibrium as the Boltzmann distribution of a continuous loop}
We consider the quantum system in $\mathbb R^d$ given by the Hamiltonian
\begin{equation}
	\hat H = \frac{\hat p^2}2 + V(\hat q),
	\label{Hamiltonian}
\end{equation}
where $\hat q$ and $\hat p$ are the position and momentum operators in $\mathbb R^d$, and $V(\cdot)$ is a real-valued potential function in $\mathbb R^d$. When the quantum system is at a constant temperature $T = 1/\beta$, the state of the system can be described by the canonical ensemble with the density operator $e^{-\beta\hat H}$, and thus the partition function $\mathcal Z = \mathrm{Tr}[e^{-\beta\hat H}]$. The quantum thermal average of the system means the \red{canonical} average of an observable operator $O(\hat q)$, namely
\begin{equation}
	\avg{O(\hat q)}_\beta
	= \frac
	{\mathrm{Tr}[e^{-\beta\hat H} O(\hat q)]}
	{\mathrm{Tr}[e^{-\beta\hat H}]},
	\label{thermal average}
\end{equation}
Here, we assume the observable operator $O(\hat q)$ depends only on the position operator $\hat q$, where $O(\cdot)$ is a real-valued function in $\mathbb R^d$.

Utilizing the Feynman path integral \cite{path,feynman}, the partition function can be interpreted as the integral with respect to a continuous loop $x(\tau)$ in $\mathbb R^d$ parameterized by $\tau\in[0,\beta]$:
\begin{equation*}
\mathcal Z = \mathrm{Tr}[e^{-\beta\hat H}] = \int \exp\{-\mathcal E(x)\}\mathcal D[x],
\end{equation*}
where $\mathcal D[x]$ is the formal Lebesgue measure of the continuous loop in $\mathbb R^d$, and the energy functional $\mathcal E(x)$ is defined by
\begin{equation}
	\mathcal E(x) = \frac12\int_0^\beta |x'(\tau)|^2\d\tau + \int_0^\beta  V(x(\tau))\d\tau.
	\label{energy Ex}
\end{equation}
As a consequence, the \red{quantum canonical ensemble} $e^{-\beta\hat H}$ is exactly mapped to a classical Boltzmann distribution of the continuous loop $x(\tau)$ with the energy $\mathcal E(x)$.
\red{Finally, the quantum thermal average $\avg{O(\hat q)}_\beta$ can be formally written as}
\begin{equation}
	\avg{O(\hat q)}_\beta = \frac1{\mathcal Z}\int \bigg[
	\frac1{\beta}\int_0^\beta O(x(\tau)) \d\tau
	\bigg]\exp\{-\mathcal E(x)\}\mathcal D[x].
	\label{O express}
\end{equation}
\subsection{Represent the energy function with \red{Matsubara} modes}
\label{section: Matsubara}
Consider the following eigenvalue problem with periodic boundary conditions:
\begin{equation*}
	-\ddot{c}_k(\tau) =
	\omega_k^2 c_k(\tau),~~~~\tau\in[0,\beta],~~~~
	k = 0,1,2,\cdots,
\end{equation*}
where the eigenvalues and eigenfunctions are explicitly given by
$$
\begin{aligned}
\omega_0 & = 0, &
c_0(\tau) & = \sqrt{\frac1\beta}; \\
\omega_{2k-1} & = \frac{2k\pi}{\beta} &
c_{2k-1}(\tau) & = \sqrt{\frac2\beta}
\sin\bigg(\frac{2k\pi\tau}{\beta}\bigg), & k = 1,2,\cdots; \\
\omega_{2k} & = \frac{2k\pi}{\beta}, &
c_{2k}(\tau) & = \sqrt{\frac2\beta}
\cos\bigg(\frac{2k\pi\tau}{\beta}\bigg), & k = 1,2,\cdots.
\end{aligned}
$$
\red{The eigenfunctions $\{c_k(\tau)\}_{k=0}^\infty$ form} the orthonormal Fourier basis of  $L^2([0,\beta];\mathbb R)$, and thus any continuous loop $x(\cdot)$ can be uniquely represented as
\begin{equation*}
	x(\tau) = \sum_{k=0}^\infty \xi_k c_k(\tau),~~~~\tau\in[0,\beta],
\end{equation*}
where $\{\xi_k\}_{k=0}^\infty$ in $\mathbb R^d$ are the coordinates of $x(\cdot)$ in different \red{Matsubara} modes. Conversely, for a given continuous loop $x(\cdot)$, the \red{Matsubara} coordinates are calculated from
\begin{equation*}
	\xi_k = \int_0^\beta x(\tau) c_k(\tau)\d\tau,
	~~~~k=0,1,2,\cdots.
\end{equation*}

Utilizing the \red{Matsubara coordinates}, the energy function $\mathcal E(x)$ in \eqref{energy Ex} is written as
\begin{equation}
	\mathcal E(\xi) = \frac12\sum_{k=0}^\infty \omega_k^2|\xi_k|^2
	+
	\mathcal V(\xi),~~~\text{with}~~\mathcal V(\xi) :=
	\int_0^\beta V\bigg(
	\sum_{k=0}^\infty \xi_k c_k(\tau)
	\bigg)\d\tau,
	\label{energy xi}
\end{equation}
and the target Boltzmann distribution is formally given by $\exp\{-\mathcal E(\xi)\}$.
\subsection{Formulation of the \red{Matsubara mode} PIMD}
In the following, we construct an infinite-dimensional Langevin  dynamics to sample the Boltzmann distribution $\exp\{-\mathcal E(\xi)\}$. \red{Without any preconditioning, the vanilla underdamped Langevin dynamics for sampling $\exp\{-\mathcal E(\xi)\}$ reads
\begin{equation}
\left\{
\begin{aligned}
\dot\xi_k & = \eta_k, \\
\dot\eta_k & = -\omega_k^2 \xi_k -
\int_0^\beta \nabla V(x(\tau)) c_k(\tau)\d\tau
- \gamma \eta_k + \sqrt{2\gamma}\dot B_k,
\end{aligned}
\right.
\label{no Langevin}
\end{equation}
where $\gamma>0$ is the damping rate on each mode, $\{\eta_k\}_{k=0}^\infty$ are the auxiliary velocity variables in $\mathbb R^d$, and $\{B_k\}_{k=0}^\infty$ are independent Brownian motions in $\mathbb R^d$.} However, the high-frequency modes in the energy function $\mathcal E(\xi)$ poses \red{the infamous stiffness problem \cite{pmmLang,md_2} in the time discretization of \eqref{no Langevin}: the time step needs to be extremely small to stabilize of the high-frequency part of the dynamics.}

In this paper, we employ the preconditioning
\cite{pmmLang,pHMC} to scale the internal frequencies of the \red{Matsubara} modes. Introduce $a>0$ and rewrite the energy function $\mathcal E(\xi)$ as
\begin{equation}
	\mathcal E(\xi) = \frac12\sum_{k=0}^\infty
	(\omega_k^2+a)|\xi_k|^2 +
	\mathcal V^a(\xi),~~~\text{with}~~
	\red{\mathcal V^a(\xi) :=
	\int_0^\beta
	V^a\bigg(\sum_{k=0}^\infty \xi_k c_k(\tau)\bigg)
	\d\tau,}
	\label{energy a}
\end{equation}
where the potential function $V^a(q) := V(q) - a|q|^2/2$.
\red{The positive coefficient $\omega_k^2+a$ on each mode indicates that it is reasonable to apply the preconditioning}
\begin{equation}
\left\{
\begin{aligned}
	\dot \xi_k & = \eta_k, \\
	\dot \eta_k & = -\xi_k -
	\frac1{\omega_k^2+a}
		\int_0^\beta \nabla V^a(x(\tau)) c_k(\tau) \d\tau
	 - \gamma \eta_k + \sqrt{\frac{2\gamma}{\omega_k^2+a}} \dot B_k,
\end{aligned}
\right.
\label{Langevin}
\end{equation}
which is an infinite-dimensional Langevin dynamics of the \red{mode coordinates $\{\xi_k\}_{k=0}^\infty$}.
\begin{remark}
If we neglect the gradient term in \eqref{Langevin}, the dynamics of $(\xi_k,\eta_k)$ reads
\begin{equation*}
	\left\{
	\begin{aligned}
	\dot \xi_k & = \eta_k, \\
	\dot \eta_k & = - \xi_k - \gamma \eta_k + \sqrt{\frac{2\gamma}{\omega_k^2+a}}\dot B_k,
	\end{aligned}
	\right.
\end{equation*}
which is a underdamped Langevin dynamics in $\mathbb R^d\times\mathbb R^d$ with the invariant distribution
\begin{equation*}
	\exp\bigg\{-\frac{\omega_k^2+a}2 \big(|\xi_k|^2+|\eta_k|^2\big)\bigg\}.
\end{equation*}
\red{Clearly, the preconditioning removes the stiffness of all Matsubara modes, and $\mathcal O(1)$ time step is applicable for the time discretization.}
\end{remark}
\begin{remark}
\red{The constant $a>0$ is to ensure the frequency of the centroid mode $(k=0)$ is nonzero, so that the preconditioning can be applied on all Matsubara modes.
An alternative preconditioning scheme for \eqref{no Langevin} without introducing $a$ is defined as:
\begin{equation}
\left\{
\begin{aligned}
\xi_0 & = \eta_0, \\
\eta_0 & = -\frac1{\sqrt\beta}\int_0^\beta \nabla V(x(\tau)) \d\tau - \gamma \eta_0 + \sqrt{2\gamma} \dot B_0,
\end{aligned}
\right.
\label{alt 1}
\end{equation}
\begin{equation}
\left\{
\begin{aligned}
\xi_k & = \eta_k, \\
\eta_k & = -\xi_k - \frac1{\omega_k^2} \int_0^\beta
\nabla V(x(\tau)) c_k(\tau) \d\tau - \gamma \eta_k +
\frac{\sqrt{2\gamma}}{\omega_k} \dot B_k,
\end{aligned}
\right.~~~~ k = 1,2,\cdots.
\label{alt 2}
\end{equation}
Moreover, the preconditioned Matsubara mode PIMD \eqref{alt 1} and \eqref{alt 2} can be employed in quantum systems defined on a periodic torus space $\mathbb T^d$.}

\end{remark}
\subsection{Implementation of the \red{Matsubara mode} PIMD in finite dimensions}
At this stage, the infinite-dimensional \red{Matsubara mode PIMD} \eqref{Langevin} is completely formal, and neither its existence or well-posedness is not justified.
For the purpose of numerical simulation, we need to truncate the number of \red{Matsubara} modes in \eqref{Langevin} to a finite integer $N$. In this case the continuous loop $x(\tau)$ is truncated as
\begin{equation*}
	x_N(\tau) = \sum_{k=0}^{N-1} \xi_k c_k(\tau),~~~~
	\tau \in [0,\beta],
\end{equation*}
and the potential energy of the loop in \eqref{energy a} is replaced by
\begin{equation}
	\red{\mathcal V_N^a(\xi)} :=
	\int_0^\beta V^a(x_N(\tau))\d\tau =
	\int_0^\beta V^a\bigg(
	\sum_{k=0}^{N-1} \xi_k c_k(\tau)
	\bigg)\d\tau,~~~~
	\xi \in \mathbb R^{dN}.
	\label{VaN}
\end{equation}
\red{However, even for a finite $N$, the integral in \eqref{VaN}} does not have an explicit expression, and the numerical integration is required to approximate $\mathcal V_N^a(\xi)$. Let $D\in\mathbb N$ be the discretization size in the interval $[0,\beta]$, then \red{we obtain the approximation}
\begin{equation}
	\red{\mathcal V_N^a(\xi) \approx  \mathcal V_{N,D}^a(\xi)} := \beta_D \sum_{j=0}^{D-1} V^a
	\bigg(\sum_{k=0}^{N-1} \xi_k c_k(j\beta_D)\bigg),
	~~~~\xi\in\mathbb R^{dN},
	\label{VaND}
\end{equation}
where $\beta_D = \beta / D$. The gradients of \red{$\mathcal V^a_{N,D}(\xi)$} with respect to each $\xi_k$ are correspondingly
\begin{equation*}
	\nabla_{\xi_k}\mathcal V_{N,D}^a (\xi) =
	\beta_D \sum_{j=0}^{N-1} \nabla V^a(x_N(j\beta_D)) c_k(j\beta_D),
	~~~~ k = 0,1,\cdots,N-1.
\end{equation*}
\begin{remark}
\red{We do not select the high-order numerical integration schemes mainly because the continuous loop $x_N(\tau)$ becomes ragged when the number of modes $N$ is large. The low-regularity of $x_N(\tau)$ makes it ineffectual to apply high-order integration.}
\end{remark}
\noindent
With the \red{potential function $\mathcal V_{N,D}^a(\xi)$}, the Matsubara mode PIMD \eqref{Langevin} is approximated as
\begin{equation}
\left\{
\begin{aligned}
\dot \xi_k & = \eta_k, \\
\dot \eta_k & = -\xi_k - \frac{
\beta_D}{\omega_k^2+a} \sum_{j=0}^{D-1} \nabla V^a(x_N(j\beta_D) )c_k(j\beta_D)
 -\gamma\eta_k + \sqrt{\frac{2\gamma}{\omega_k^2+a}}\dot B_k,
\end{aligned}
\right.~~~k=0,1,\cdots,N-1,
\label{under a}
\end{equation}
whose invariant distribution is the classical Boltzmann distribution:
\begin{equation}
\pi_{N,D}(\xi) \propto \exp\bigg\{
-\frac12\sum_{k=0}^{N-1} (\omega_k^2+a)|\xi_k|^2
-\beta_D \sum_{j=0}^{D-1} V^a
\bigg(\sum_{k=0}^{N-1} \xi_k c_k(j\beta_D)\bigg)
\bigg\}.
\label{pi N}
\end{equation}
A typical choice of the discretization size $D$ is the number of modes $N$, so that the \red{computational cost per time step is equal to $\mathcal O(N\log N)$ in the Fast Fourier Transform}. It is natural from \eqref{O express} that the quantum thermal average $\avg{O(\hat q)}_\beta$ is approximated as
\begin{equation*}
	\avg{O(\hat q)}_\beta \approx
	\avg{O(\hat q)}_{\beta,N,D} =
	\int_{\mathbb R^{dN}}
	\bigg[
	\frac1D\sum_{j=0}^{D-1}
	O\bigg(\sum_{k=0}^{N-1} \xi_k c_k(j\beta_D)\bigg)
	\bigg]\pi_{N,D}(\xi)\d\xi.
\end{equation*}
\red{The long-time simulation of the Matsubara mode PIMD \eqref{under a} at finite $N$ and $D$ then provides an accurate estimate of the statistical average in the distribution $\pi_{N,D}(\xi)$, which is denoted by $\avg{O(\hat q)}_{\beta,N,D}$.}
\subsection{Relation to the standard PIMD}
\red{We demonstrate the connection between the Matsubara mode PIMD and the standard PIMD. Suppose the number of modes $N$ is an odd integer and we view $\{\xi_k\}_{k=0}^{N-1}$ as the normal mode coordinates \cite{simple} in the standard PIMD, then the Boltzmann distribution of the ring polymer beads in the standard PIMD is expressed as
\begin{equation}
\exp\bigg\{
-\frac12\sum_{k=0}^{N-1} (\omega_{k,N}^2+a) |\xi_k|^2 - \beta_N
\sum_{j=0}^{N-1}V^a\bigg(\sum_{k=0}^{N-1} \xi_k c_k(j\beta_N)\bigg)
\bigg\},
\label{under N}
\end{equation}
where $\{\omega_{k,N}\}_{k=0}^{N-1}$ are the normal mode frequencies
\begin{equation*}
	\omega_0 = 0,~~~~
	\omega_{2k-1,N} = \omega_{2k,N} =
	\frac{2}{\beta_N}  \sin\bigg(\frac{k\pi}N\bigg),
	~~~~ k = 1,\cdots,\frac{N-1}2.
\end{equation*}
Furthermore, the preconditioned Langevin dynamics for sampling \eqref{under N} is written as
\begin{equation}
\left\{
\begin{aligned}
\dot \xi_k & = \eta_k, \\
\dot \eta_k & = -\xi_k - \frac{
\beta_D}{\omega_{k,N}^2+a} \sum_{j=0}^{D-1} \nabla V^a(x_N(j\beta_D) )c_k(j\beta_D)
 -\gamma\eta_k + \sqrt{\frac{2\gamma}{\omega_{k,N}^2+a}}\dot B_k,
\end{aligned}
\right.~~~k=0,1,\cdots,N-1.
\label{under std}
\end{equation}

Comparing the dynamics \eqref{under a} and \eqref{under std}, we observe the standard PIMD and the Matsubara mode PIMD (with odd $D=N$) are equivalent except for the mode frequencies.
Since the normal mode frequencies satisfy
\begin{equation*}
	\lim_{N\rightarrow\infty} \omega_{k,N} = \omega_k,~~~~k = 0,1,2,\cdots,
\end{equation*}
we conclude that the Matsubara mode PIMD \eqref{under a} and the
standard PIMD \eqref{under std} have the identical continuum limit, namely the infinite-dimensional Langevin dynamics \eqref{Langevin}.
See Appendix~C for the detailed derivation of the normal mode coordinates.}
\section{Uniform-in-$N$ ergodicity of the Matsubara mode PIMD}
\label{section: convergence analysis}
\red{In this section, we prove the uniform-in-$N$ ergodicity of the Matsubara mode PIMD. First, we present the assumptions and ergodicity results in Section~3.1. Then, we prove the uniform-in-$N$ ergodicity in the overdamped and underdamped cases in Sections~3.2 and 3.3, respectively. Finally, in Section~3.4, we demonstrate that our proof of the uniform-in-$N$ ergodicity in the Matsubara mode PIMD applies to the standard PIMD.}
\subsection{Assumptions and ergodicity results}
\red{For the convenience of analysis, in the Matsubara mode PIMD \eqref{under a} we choose the damping rate $\gamma = 1$, yielding the following underdamped Langevin dynamics:}
\begin{equation}
	\left\{
	\begin{aligned}
	\dot \xi_k & = \eta_k, \\
	\dot \eta_k & = -\xi_k -
	\frac{\beta_D}{\omega_k^2+a} \sum_{j=0}^{D-1} \nabla V^a(x_N(j\beta_D)) c_k(j\beta_D) - \eta_k +
	\sqrt{\frac2{\omega_k^2+a}}\dot B_k.
	\end{aligned}
	\right.
	\label{under}
\end{equation}
\red{A closely related dynamics is the overdamped version of \eqref{under}:}
\begin{equation}
	\dot \xi_k = -\xi_k -
	\frac{\beta_D}{\omega_k^2+a}
	\sum_{j=0}^{D-1} \nabla V^a(x_N(j\beta_D)) c_k(j\beta_D)
	+ \sqrt{\frac2{\omega_k^2+a}}\dot B_k,
	\label{over}
\end{equation}
which can be viewed as the overdamped limit of \eqref{under} as $\gamma\rightarrow\infty$ \cite{explicit}.
\red{The invariant distribution of \eqref{over} is $\pi_{N,D}(\xi)$ defined in \eqref{pi N}, and the invariant distribution of \eqref{under} is}
\begin{equation}
\mu_{N,D}(\xi,\eta) \propto
\exp\bigg(-
\frac12\sum_{k=0}^{N-1}
(\omega_k^2+a) (|\xi_k|^2+|\eta_k|^2) -
\beta_D\sum_{j=0}^{D-1} V^a\bigg(
\sum_{k=0}^{N-1} \xi_k c_k(j\beta_D)
\bigg)
\bigg).
\label{mu N}
\end{equation}
\red{Clearly, $\pi_{N,D}(\xi)$ is the marginal distribution of $\mu_{N,D}(\xi)$ in the variable $\xi\in\mathbb R^{dN}$.}

Under appropriate conditions on the potential $V(q)$, we prove that both the \red{overdamped and underdamped Matsubara mode PIMD in \eqref{over} and \eqref{under}} possess the uniform-in-$N$ ergodicity, namely, the convergence rate to the invariant distribution does not depend on the number of modes $N$.

Before we conduct a detailed discussion on these results, we enumerate all the assumptions required in the proof, mainly on the potential function $V(q)$ in $\mathbb R^d$.
\begin{assumption}
Given $a>0$, the potential function
$$
	V^a(q) = V(q) - \frac{a}2|q|^2,~~~~q\in\mathbb R^d
$$
is twice differentiable in $\mathbb R^d$, and for some constants $M_1,M_2 \Ge0$:
\begin{enumerate}
\item[(i)] $V^a(q)$ can be decomposed as $V^c(q) + V^b(q)$, where $\nabla^2 V^c(q)\Succ O_d$ and $|V^b(q)|\Le M_1$ for any $q\in\mathbb R^d$.
\item[(ii)] $-M_2I_d \Prec \nabla^2 V^a(q) \Prec  M_2I_d$ for any $q\in\mathbb R^d$.
\end{enumerate}
\red{A special assumption required in the underdamped case is:
\begin{enumerate}
\item[(iii)] The number of modes $N$ is no larger than the discretization size $D$, namely, $N\Le D$.
\end{enumerate}
Here, $I_d$ and $O_d$ are the identity and zero matrix in $\mathbb R^{d\times d}$. Assumption (i) can be shortly interpreted as: $V^c(q)$ is globally convex and $V^b(q)$ is globally bounded.}
\end{assumption}

Next, we display the ergodicity results of the \red{Matsubara mode PIMD} in Table \ref{table: result 1}.
\begin{center}
\begin{tabular}{c|c|c}
\toprule
\multicolumn{3}{c}{
\textbf{Uniform-in-$N$ ergodicity of the \red{Matsubara mode} PIMD}
} \\
\midrule
dynamics & overdamped \eqref{over} & underdamped \eqref{under} \\
\midrule
assumption
& (i) $\Rightarrow$ Theorem~\ref{theorem: over ergodicity}
& (i)(ii)(iii) $\Rightarrow$ Theorem~\ref{theorem: under ergodicity} \\
\midrule
distribution & $\pi_{N,D}(\xi)$ in \eqref{pi N} & $\mu_{N,D}(\xi,\eta)$ in \eqref{mu N} \\
\midrule
ergodicity &
$\Ent_{\pi_{N,D}}(P_t f) \Le e^{-2\lambda_1 t} \Ent_{\pi_{N,D}}(f)$ &
$W_{\mu_{N,D}}(P_tf) \Le e^{-2\lambda_2 t} W_{\mu_{N,D}}(f)$ \\
\bottomrule
\end{tabular}
\captionof{table}{The uniform-in-$N$ ergodicity of the Matsubara mode PIMD. The relative entropy $\Ent_{\pi_{N,D}}(f)$ and the entropy-like quantity $W_{\mu_{N,D}}(f)$ are defined in \eqref{entropy_1} and \eqref{entropy_2}.}
\label{table: result 1}
\end{center}
\red{In Table~\ref{table: result 1}, $\Ent_{\pi_{N,D}}(f)$ is the relative entropy of the density function $f(\xi)$ in $\mathbb R^{dN}$:
\begin{equation}
	\Ent_{\pi_{N,D}}(f) := \int_{\mathbb R^{dN}} f\log f\d\pi_{N,D} -
	\int_{\mathbb R^{dN}} f\d\pi_{N,D} \log \int_{\mathbb R^{dN}} f\d\pi_{N,D},
	\label{entropy_1}
\end{equation}
and $W_{\mu_{N,D}}(f)$ is the entropy-like quantity of the density function $f(\xi,\eta)$ in $\mathbb R^{2dN}$:
\begin{equation}
	W_{\mu_{N,D}}(f):=\bigg(\frac{M_2^2}{a^2}+1\bigg) \Ent_{\mu_{N,D}}(f) +
		\sum_{k=0}^{N-1} \frac1{\omega_k^2+a}
		\int_{\mathbb R^{2dN}}
		\frac{|\nabla_{\eta_k}f - \nabla_{\xi_k}f|^2 + |\nabla_{\eta_k}f|^2}{f}\d\mu_{N,D}.
	\label{entropy_2}
\end{equation}
Furthermore,} the convergence rates $\lambda_1$ and $\lambda_2$ are given by
\begin{equation*}
	\lambda_1  = \exp(-4\beta M_1),~~~~
	\lambda_2 = \frac{a^2}{3M_2^2+5a^2} \exp(-4\beta M_1),
\end{equation*}
\begin{remark}
The convergence rates $\lambda_2<\lambda_1$, and they are exponentially small in the low temperature limit $\beta\rightarrow\infty$ if the constant $M_1>0$, which characterizes the non-convexity of the potential $V(q)$.
This is because we utilize the bounded perturbation of the log-Sobolev inequalities (Proposition~5.1.6 of \cite{bakry}).
However, $\lambda_2 < \lambda_1$ does not imply the \red{overdamped Matsubara mode PIMD} converges faster than the underdamped one in practical simulation. With the hypocoercivity methods, it is proved that introducing auxiliary velocity variables accelerates the convergence of the overdamped Langevin dynamics \cite{explicit}.
\end{remark}
\begin{remark}
\red{Assumption (iii) comes from the discrete orthogonal condition \eqref{discrete normalize}:
\begin{equation*}
	\sum_{j=0}^{D-1} c_k(j\beta_D) c_l(j\beta_D) = 0,~~~~
	0\Le k<l \Le N-1,
\end{equation*}
which in general is incorrect when $N>D$. The authors conjecture that the uniform-in-$N$ ergodicity of the underdamped Matsubara mode PIMD \eqref{under} also holds true when $N>D$, nevertheless no viable strategy is available.}
\end{remark}
\subsection{Uniform ergodicity of overdamped Matsubara mode PIMD}
\label{section: overdamped}
We prove the uniform-in-$N$ ergodicity of \eqref{over} in the relative entropy using the  Bakry--\'Emery calculus \cite{bakry}.
The log-Sobolev inequality \red{for the distribution $\pi_{N,D}(\xi)$} produces an explicit convergence rate in high-dimensions.
\begin{theorem}
\label{theorem: over ergodicity}
Under Assumption (i), let $(P_t)_{t\Ge0}$ be the Markov semigroup of the \red{overdamped Matsubara mode} PIMD \eqref{over}, then for any positive smooth function $f(\xi)$ in $\mathbb R^{dN}$,
\begin{equation*}
	\Ent_{\pi_{N,D}}(P_t f) \Le \exp(-2\lambda_1 t)
	\Ent_{\pi_{N,D}}(f),~~~~\forall t\Ge0,
\end{equation*}
where the convergence rate $\lambda_1 = \exp(-4\beta M_1)$.
\end{theorem}
\begin{proof}
We study the overdamped \red{Matsubara mode PIMD} driven by the convex potential $V^c(q)$, and prove the corresponding uniform-in-$N$ log-Sobolev inequality. Utilizing the bounded perturbation, we obtain the uniform-in-$N$ log-Sobolev inequality for \eqref{over}.\\[6pt]
\textbf{1. Ergodicity of overdamped Matsubara mode PIMD driven by $V^c(q)$}\\[6pt]
For notational convenience, introduce the potential function of the convex part $V^c(q)$:
\begin{equation*}
	\mathcal V_{N,D}^c(\xi) := \beta_D \sum_{j=0}^{D-1} V^c(x_N(j\beta_D)) = \beta_D \sum_{j=0}^{D-1}
	V^c\bigg(\sum_{k=0}^{N-1} \xi_k c_k(j\beta_D)\bigg),
\end{equation*}
Since $V^c(q)$ is globally convex in $\mathbb R^d$, it is easy to deduce $\mathcal V_{N,D}^c(\xi)$ is globally convex in the mode coordinates $\xi = \{\xi_k\}_{k=0}^{N-1}$ \red{(see Lemma~\ref{appendix: 1} in Appendix~\ref{appendix: proof convergence})}.
Next consider \red{the overdamped Matsubara mode PIMD} driven by $\mathcal V_{N,D}^c(\xi)$:
\begin{equation}
	\dot \xi_k = -\xi_k - \frac1{\omega_k^2+a}
	\nabla_{\xi_k} \mathcal V_{N,D}^c(\xi) +
	\sqrt{\frac2{\omega_k^2+a}}\dot B_k,~~~~
	k=0,1,\cdots,N-1,
	\label{over_c}
\end{equation}
It is easy to see the generator of \eqref{over_c} is given by
\begin{equation*}
	L^c = - \sum_{k=0}^{N-1}
	\bigg(\xi_k +
	\frac1{\omega_k^2+a} \nabla_{\xi_k} \mathcal V_{N,D}^c(\xi)
	\bigg)\cdot \nabla_{\xi_k} +
	\sum_{k=0}^{N-1} \frac1{\omega_k^2+a} \Delta_{\xi_k},
\end{equation*}
and the invariant distribution of \eqref{over_c} is
\begin{equation}
	\pi_{N,D}^c(\xi) \propto
	\exp\bigg\{
		-\frac12\sum_{k=0}^{N-1}
		(\omega_k^2+a) |\xi_k|^2
		- \beta_D \sum_{j=0}^{D-1}
			V^c\bigg(\sum_{k=0}^{N-1} \xi_k c_k(j\beta_D)\bigg)
	\bigg\}.
	\label{pi Nc}
\end{equation}

To establish the log-Sobolev inequality for the distribution $\pi_{N,D}^c(\xi)$, we compute the carr\'e du champ operator $\Gamma_1(f,f)$ and the iterated operator $\Gamma_2(f,f)$ corresponding to the generator $L^c$ (see Definition 1.4.2 and Equation (1.16.1) in \cite{bakry}). By direct calculation,
\begin{align*}
\Gamma_1(f,g) & = \sum_{k=0}^{N-1}
\frac{\nabla_{\xi_k}f\cdot \nabla_{\xi_k} g}{\omega_k^2+a}, \\
	\Gamma_2(f,g) & =
	\sum_{k,l=0}^{N-1} \frac{\nabla_{\xi_k\xi_l}^2 f :
		\nabla_{\xi_k\xi_l}^2 g}{(\omega_k^2+a)(\omega_l^2+a)}
	 + \sum_{k=0}^{N-1}
	\frac{\nabla_{\xi_k} f \cdot\nabla_{\xi_k} g}{\omega_k^2+a}
	  +
	\sum_{k,l=0}^{N-1} \frac{\nabla_{\xi_k} f\cdot \nabla^2_{\xi_k\xi_l} \mathcal V_{N,D}^c(\xi)\cdot \nabla_{\xi_l} g}{(\omega_k^2+a)(\omega_l^2+a)}.
\end{align*}
Here, the dot product $u\cdot B\cdot v$ for $u\in\mathbb R^d$, $B\in\mathbb R^{d\times d}$ and $v\in\mathbb R^d$ means
\begin{equation*}
	u\cdot B\cdot v = u^\T B v = \sum_{p,q=1}^d
	u_p B_{pq} v_q,
\end{equation*}
and the double dot product $A:B$ for $A\in\mathbb R^{d\times d}$ and $B\in\mathbb R^{d\times d}$ means
\begin{equation*}
	A:B = \mathrm{Tr}[A^\T B] = \sum_{p,q=1}^d
	A_{pq} B_{pq}.
\end{equation*}
\red{Utilizing the convexity of the potential function $\mathcal V_{N,D}^c(\xi)$}, we obtain
\begin{equation*}
\Gamma_2(f,f) \Ge
\sum_{k=0}^{N-1} \frac{|\nabla_{\xi_k} f|^2}{\omega_k^2+a}  + \sum_{k,j=0}^{N-1} \frac{\nabla_{\xi_k} f\cdot \nabla^2_{\xi_k\xi_j} \mathcal V_{N,D}^c(\xi)\cdot \nabla_{\xi_j} f}{(\omega_k^2+a)(\omega_j^2+a)} \Ge \sum_{k=0}^{N-1}
\frac{|\nabla_{\xi_k}f|^2}{\omega_k^2+a} = \Gamma_1(f,f),
\end{equation*}
hence the log-Sobolev constant for $\pi_{N,D}^c(\xi)$ is 1 according to Proposition 5.7.1 of \cite{bakry}. The log-Sobolev inequality for $\pi_{N,D}^c(\xi)$ then reads
\begin{equation}
	\Ent_{\pi_{N,D}^c}(f) \Le \frac12
	\int_{\mathbb R^{dN}} \frac{\Gamma_1(f,f)}{f}
	\d\pi_{N,D}^c = \frac12
	\sum_{k=0}^{N-1} \frac1{\omega_k^2+a}
	\int_{\mathbb R^{dN}}
	\frac{|\nabla_{\xi_k} f|^2}{f} \d\pi_{N,D}^c.
	\label{LS_c}
\end{equation}
\red{Note that the relative entropy here is defined with respect to the distribution $\pi_{N,D}^c(\xi)$ with $V^c(q)$ rather than the distribution $\pi_{N,D}(\xi)$ with $V^a(q)$.}\\[6pt]
\textbf{3. Ergodicity of overdamped Matsubara mode PIMD driven by $V^a(q)$}\\[6pt]
Let $Z_{N,D}$ and $Z_{N,D}^c$ be the normalization constants of the distributions $\pi_{N,D}(\xi)$ and $\pi_{N,D}^c(\xi)$,
\begin{equation*}
\begin{aligned}
Z_{N,D} & = \int_{\mathbb R^{dN}}\exp\bigg(-\frac12\sum_{k=0}^{N-1} (\omega_k^2+a) |\xi_k|^2 - \beta_D \sum_{j=0}^{D-1} V^a(x_N(j\beta_D))\bigg)
\d\xi, \\
Z_{N,D}^c & = \int_{\mathbb R^{dN}}\exp\bigg(-\frac12\sum_{k=0}^{N-1} (\omega_k^2+a) |\xi_k|^2 - \beta_D \sum_{j=0}^{D-1} V^c(x_N(\beta_D))\bigg)
\d\xi.
\end{aligned}
\end{equation*}
\red{Using the inequality $|V^a(q) - V^c(q)| = |V^b(q)| \Le M_1$, we have
\begin{equation*}
	|\mathcal V_{N,D}^a(\xi) - \mathcal V_{N,D}^c(\xi)| \Le \beta_D\sum_{j=0}^{D-1} \big|V^b(x_N(j\beta_D))\big| \Le \beta M_1,
\end{equation*}
and thus the constants $Z_{N,D}$ and $Z_{N,D}^c$ satisfy}
\begin{equation*}
	\frac{Z_{N,D}}{Z_{N,D}^c} \in \big[\exp(-\beta M_1),\exp(\beta M_1)\big].
\end{equation*}
As a result, the density functions $\pi_{N,D}(\xi)$ and $\pi_{N,D}^c(\xi)$ satisfy
\begin{equation*}
	\frac{\pi_{N,D}^c(\xi)}{\pi_{N,D}(\xi)} =
	\frac{Z_{N,D}}{Z_{N,D}^c} \exp\bigg(
	\beta_D \sum_{j=0}^{D-1} V^b(x_N(j\beta_D))
	\bigg) \in \big[\exp(-2\beta M_1), \exp(2\beta M_1)\big].
\end{equation*}
Using the bounded perturbation (Proposition 5.1.6 of \cite{bakry}),
we obtain from  \eqref{LS_c} that
\begin{equation*}
	\exp(-4\beta M_1) \Ent_{\pi_{N,D}}(f) \Le \frac12
	\sum_{k=0}^{N-1} \frac1{\omega_k^2+a}
	\int_{\mathbb R^{dN}}
	\frac{|\nabla_{\xi_k} f|^2}{f} \d\pi_{N,D}.
\end{equation*}
Hence for the rate $\lambda_1 = \exp(-4\beta M_1)$, the relative entropy has exponential decay,
\begin{equation*}
	\Ent_{\pi_{N,D}}(P_t f) \Le \exp(-2\lambda_1 t)
	\Ent_{\pi_{N,D}}(f),~~~~\forall t\Ge0,
\end{equation*}
for any positive smooth function $f(\xi)$.
\end{proof}
\red{In particular, the convergence rate $\lambda_1$ does not depend on the number of modes $N$ or the discretization size $D$, hence we conclude the uniform-in-$N$ ergodicity of the overdamped Matsubara mode PIMD \eqref{over}.}
\subsection{Uniform ergodicity of underdamped Matsubara mode PIMD}
\label{section: underdamped}
We prove the uniform-in-$N$ ergodicity of \eqref{under} in the \red{entropy-like quantity \eqref{entropy_2}}, and the main technique is the generalized $\Gamma$ calculus developed in \cite{m1,m2}. The generalized $\Gamma$ calculus is an extension of the Bakry--\'Emery calculus and can be applied on stochastic processes with degenerate diffusions. The generalized $\Gamma$ calculus is largely inspired from the hypocoercivity theory \cite{villani} of Villani, and is able to produce an explicit convergence rate in the relative entropy rather than $H^1$ or $L^2$.
For convenience, we present a brief review of the theory of the generalized $\Gamma$ calculus in Appendix \ref{appendix: review gamma}.

Recall that \red{the invariant distribution of the underdamped Matsubara mode PIMD \eqref{under} is} $\mu_{N,D}(\xi,\eta)$ defined in \eqref{mu N}.
\begin{theorem}
\label{theorem: under ergodicity}
Under Assumptions (i)(ii)(iii), let $(P_t)_{t\Ge0}$ be the Markov semigroup of the underdamped \red{Matsubara mode PIMD} \eqref{under}, then for any positive smooth function $f(\xi,\eta)$ in $\mathbb R^{2dN}$,
\begin{equation*}
	W_{\mu_{N,D}}(P_tf) \Le \exp(-2\lambda_2 t) W_{\mu_{N,D}}(f),~~~~
	\forall t\Ge0,
\end{equation*}
where the convergence rate $\lambda_2 = \frac{a^2}{3M_2^2+5a^2} \exp(-4\beta M_1)$.
\end{theorem}
\begin{proof}
The proof is accomplished in several steps.
First, we establish the uniform-in-$N$ log-Sobolev inequality for the distribution $\mu_{N,D}(\xi,\eta)$. Second, introduce the functions
\begin{equation*}
	\Phi_1(f) = f\log f,~~~~
	\Phi_2(f) = \sum_{k=0}^{N-1} \frac1{\omega_k^2+a}
	\frac{|\nabla_{\eta_k}f - \nabla_{\xi_k}f|^2 + |\nabla_{\eta_k}f|^2}f
\end{equation*}
\red{and compute the generalized $\Gamma$ operators $\Gamma_{\Phi_1}(f)$ and $\Gamma_{\Phi_2}(f)$.}
Finally, by validating the generalized curvature-dimension condition
\begin{equation*}
	\Gamma_{\Phi_2}(f) - \frac12 \Phi_2(f) +
	\bigg(\frac{M_2^2}{a^2}+1\bigg)\Gamma_{\Phi_1}(f) \Ge 0,
\end{equation*}
we can apply Theorem~\ref{theorem: generalized} to prove the exponential decay of the quantity $W(P_t f)$.\\[6pt]
\textbf{1. Uniform-in-$N$ log-Sobolev inequality for $\mu_N(\xi,\eta)$}\\[6pt]
In Theorem~\ref{theorem: over ergodicity}, the log-Sobolev inequality for the distribution $\pi_{N,D}(\xi)$ holds true:
\begin{equation*}
	\lambda_1 \Ent_{\pi_{N,D}}(f) \Le \frac12
	\sum_{k=0}^{N-1} \frac1{\omega_k^2+a}
	\int_{\mathbb R^{dN}} \frac{|\nabla_{\xi_k} f|^2}{f}\d\pi_{N,D},
\end{equation*}
where the convergence rate $\lambda_1 = \exp(-4\beta M_1)$. For the velocity variables $\{\eta_k\}_{k=0}^{N-1}$ in $\mathbb R^{dN}$, define the Gaussian distribution $\nu_N(\eta)$ by its density function:
\begin{equation*}
	\nu_N(\eta) \propto
	\exp\bigg(-\frac12\sum_{k=0}^{N-1}
	(\omega_k^2+a)|\eta_k|^2\bigg),
	~~~~ \eta \in \mathbb R^{dN},
\end{equation*}
then log-Sobolev inequality for $\nu_N$ holds true,
\begin{equation*}
	\Ent_{\nu_N}(f) \Le \frac12
	\sum_{k=0}^{N-1} \frac1{\omega_k^2+a}
	\int_{\mathbb R^{dN}} \frac{|\nabla_{\eta_k}f|^2}f\d\nu_N.
	\label{LS eta}
\end{equation*}
Since the distribution $\mu_{N,D}(\xi,\eta) = \pi_{N,D}(\xi) \otimes \nu_N(\eta)$ is the tensor product, Proposition~5.2.7 of \cite{bakry} yields the \red{log-Sobolev inequality for the product distribution:}
\begin{equation}
	\lambda_1\mathrm{Ent}_{\mu_{N,D}}(f)  \Le \frac12 \sum_{k=0}^{N-1}
	\frac1{\omega_k^2+a}\int_{\mathbb R^{2dN}} \frac{|\nabla_{\xi_k}f|^2+|\nabla_{\eta_k}f|^2}{f}\d\mu_{N,D},
	\label{LS}
\end{equation}
where the convergence rate is determined by the smaller one of the rates $\lambda_1 = \exp(-4\beta M_1)$ and $1$, which is $\lambda_1$ itself.
Note that \eqref{LS} does not imply the ergodicity of  \eqref{under} directly, because \eqref{under} has degenerate diffusion in the $\eta$ variable.\\[6pt]
\textbf{2. Calculation of the generalized $\Gamma$ operators for $\Phi_1(f)$ and $\Phi_2(f)$}\\[6pt]
We still use the notation
\begin{equation*}
\mathcal V_{N,D}^a(\xi) = \beta_D\sum_{j=0}^{D-1} V^a\bigg(\sum_{k=0}^{N-1} \xi_k c_k(j\beta_D)\bigg),~~~~\xi\in\mathbb R^{dN},
\end{equation*}
and the generator of \eqref{under} is given by
$$
	L = \sum_{k=0}^{N-1} \eta_k \cdot \nabla_{\xi_k}
	- \sum_{k=0}^{N-1} \bigg(
	\xi_k + \eta_k + \frac1{\omega_k^2+a}
	\nabla_{\xi_k} \mathcal V_{N,D}^a(\xi)
	\bigg) \cdot \nabla_{\eta_k} + \sum_{k=0}^{N-1}
	\frac{\Delta_{\xi_k}}{\omega_k^2+a}.
$$
By direct calculation, the commutators $[L,\nabla_{\xi_k}]$ and $[L,\nabla_{\eta_k}]$ are given by
\begin{equation*}
	[L,\nabla_{\xi_k}] = \nabla_{\eta_k} + \sum_{l=0}^{N-1}
	\frac1{\omega_l^2+a} \nabla_{\xi_k\xi_l}^2
	\mathcal V_{N,D}^a(\xi)\cdot \nabla_{\eta_l},~~~~
	[L,\nabla_{\eta_k}] = \nabla_{\eta_k} - \nabla_{\xi_k}.
\end{equation*}
Inspired from Example~3 of \cite{m2},
define the functions
\begin{equation}
	\Phi_1(f) = f\log f,~~~
	\Phi_2(f) = \sum_{k=0}^{N-1} \frac1{\omega_k^2+a}
	\frac{|\nabla_{\eta_k}f - \nabla_{\xi_k} f|^2 + |\nabla_{\eta_k}f|^2}f,
\end{equation}
where $\Phi_2(f)$ can be viewed as a twisted form of
\begin{equation*}
	\sum_{k=0}^{N-1} \frac1{\omega_k^2+a}
	\frac{|\nabla_{\xi_k}f|^2 + |\nabla_{\eta_k}f|^2}{f},
\end{equation*}
which appears in the RHS of the log-Sobolev inequality \eqref{LS}.
From Example~\ref{example: 2} in Appendix \ref{appendix: review gamma}, the generalized $\Gamma$ operator $\Gamma_{\Phi_1}(f)$ is given by
\begin{equation}
	\Gamma_{\Phi_1}(f) = \frac12\sum_{k=0}^{N-1}
	\frac1{\omega_k^2+a} \frac{|\nabla_{\eta_k}f|^2}{f}.
\end{equation}
To compute $\Gamma_{\Phi_2}(f)$, we write $\Phi_2(f) = \sum_{k=0}^{N-1} \frac1{\omega_k^2+a}
	\Phi_{2,k}(f)$, where
\begin{equation*}
	\Phi_{2,k}(f) = \frac{|\nabla_{\eta_k}f - \nabla_{\xi_k} f|^2 + |\nabla_{\eta_k}f|^2}f,~~~~ k =0,1,\cdots,N-1.
\end{equation*}
According to Example \ref{example: 4} in Appendix \ref{appendix: review gamma}, we have
\begin{align*}
 f\cdot \Gamma_{\Phi_{2,k}}(f) & \Ge
	(\nabla_{\eta_k}f - \nabla_{\xi_k}f) \cdot
	[L,\nabla_{\eta_k} - \nabla_{\xi_k}]f +
	\nabla_{\eta_k}f \cdot [L,\nabla_{\eta_k}] f \\
	& = |\nabla_{\eta_k}f - \nabla_{\xi_k}f|^2 -
	(\nabla_{\eta_k}f - \nabla_{\xi_k}f) \cdot
	\sum_{l=0}^{N-1} \frac1{\omega_l^2+a}
	\nabla_{\xi_k\xi_l}^2 \mathcal V_{N,D}^a(\xi) \cdot \nabla_{\eta_l} f.
\end{align*}
Taking the summation over $k = 0,1,\cdots,N-1$, we obtain
\begin{equation}
	f\cdot \Gamma_{\Phi_2}(f)
\Ge \sum_{k=0}^{N-1} \frac{|\nabla_{\eta_k} f - \nabla_{\xi_k} f|^2}{\omega_k^2+a} -
\sum_{k,l=0}^{N-1} \frac{\nabla_{\eta_k} f - \nabla_{\xi_k}f}{\omega_k^2+a}
\cdot \nabla_{\xi_k\xi_l}^2 \mathcal V_{N,D}^a(\xi) \cdot
\frac{\nabla_{\eta_l}f}{\omega_l^2+a}.
\label{Gamma 2 bound}
\end{equation}
\red{To further simplify the expression of $\Gamma_{\Phi_2}(f)$,}
define the vectors $X,Y\in\mathbb R^{dN}$ by
\begin{equation}
	X = \bigg\{\frac{\nabla_{\eta_k}f - \nabla_{\xi_k}f}{
	\sqrt{\omega_k^2+a}} \bigg\}_{k=0}^{N-1} \in \mathbb R^{dN},~~~~
	Y = \bigg\{
	\frac{\nabla_{\eta_k}f}{\sqrt{\omega_k^2+a}}
	\bigg\}_{k=0}^{N-1} \in \mathbb R^{dN},
	\label{XY}
\end{equation}
then the inequality \eqref{Gamma 2 bound} can be equivalently written as
\begin{equation}
	\Gamma_{\Phi_2}(f) \Ge \frac{|X|^2 - X^\T \Sigma Y}{f},
	\label{Gamma 2 ge}
\end{equation}
where the symmetric matrix $\Sigma\in \mathbb R^{dN\times dN}$ is given by
\begin{equation*}
	\Sigma_{kl} =
	 \frac1{
	 \sqrt{(\omega_k^2+a)
	 (\omega_l^2+a)}}
	 \nabla_{\xi_k\xi_l}^2 \mathcal V_{N,D}^a(\xi)
	 \in \mathbb R^{d\times d},~~~~
	 k,l = 0,1,\cdots,N-1.
\end{equation*}
\red{By Lemma~\ref{appendix: 2} we have} $-\frac{M_2}{a} I_{dN} \Prec \Sigma \Prec \frac{M_2}{a} I_{dN}$, hence
\eqref{Gamma 2 ge} directly produces
\begin{equation}
	\Gamma_{\Phi_2}(f) \Ge \frac{|X|^2 - \frac{M_2}{a}|X||Y|}{f}.
	\label{Gamma 2 new}
\end{equation}
In conclusion, the functions $\Phi_1(f)$, $\Phi_2(f)$ and their generalized $\Gamma$ operators satisfy
\begin{equation*}
	\Phi_1(f) = f\log f,~~~~
	\Phi_2(f) = \frac{|X|^2+|Y|^2}f,~~~~
	\Gamma_{\Phi_1}(f)= \frac{|Y|^2}{2f},~~~~
	\Gamma_{\Phi_2}(f) \Ge \frac{|X|^2 - \frac{M_2}{a}|X||Y|}{f},
\end{equation*}
where the vectors $X,Y\in\mathbb R^{dN}$ are given in \eqref{XY}.\\[6pt]
\textbf{3. Generalized curvature-dimension condition produces ergodicity}\\[6pt]
Let us summarize the functional inequalities. The log-Sobolev inequality \eqref{LS} implies
\begin{equation*}
	\lambda_1\Ent_{\mu_{N,D}}(f) \Le \frac12
	\int_{\mathbb R^{2dN}} \frac{|X+Y|^2+|Y|^2}f\d\mu_{N,D} \Le
	\frac32 \int_{\mathbb R^{2dN}} \frac{|X|^2 + |Y|^2}f\d\mu_{N,D},
\end{equation*}
hence with the expressions of $\Phi_1(f)$ and $\Phi_2(f)$, we can equivalently write
\begin{equation}
\frac{2\lambda_1}3 \bigg(\int_{\mathbb R^{2dN}}
\Phi_1(f)\d\mu_{N,D} - \Phi_1\bigg(
\int_{\mathbb R^{2dN}} f\d\mu_{N,D}
\bigg)\bigg) \Le \int_{\mathbb R^{2dN}}
\Phi_2(f) \d\mu_{N,D}.
\label{functional 1}
\end{equation}
The inequality \eqref{Gamma 2 new} implies
\begin{equation*}
	\Gamma_{\Phi_2}(f) - \frac12 \Phi_2(f) \Ge
	\frac{|X|^2 - 2\frac{M_2}{a}|X||Y| - |Y|^2}{2f},
\end{equation*}
and thus using the expression of $\Gamma_{\Phi_1}(f)$ we have
\begin{equation}
	\Gamma_{\Phi_2}(f) - \frac12\Phi_2(f) +
	\bigg(\frac{M_2^2}{a^2}+1\bigg)
	\Gamma_{\Phi_1}(f) \Ge
	\frac{(|X| - \frac{M_2}{a}|Y|)^2}{2f} \Ge 0.
	\label{functional 2}
\end{equation}
Utilizing Theorem \ref{theorem: generalized} in Appendix \ref{appendix: review gamma}, define the entropy-like quantity  by
\begin{equation*}
	W_{\mu_{N,D}}(f) = \bigg(\frac{M_2^2}{a^2}+1\bigg)
	\bigg(
	\int_{\mathbb R^{2dN}}
	\Phi_1(f)\d\mu_{N,D} - \Phi_1\bigg(
	\int_{\mathbb R^{2dN}} f\d\mu_{N,D}\bigg)
	\bigg) + \int_{\mathbb R^{2dN}}
	\Phi_2(f)\d\mu_{N,D}.
\end{equation*}
From the functional inequalities \eqref{functional 1} and \eqref{functional 2}, we derive the exponential decay
\begin{equation*}
	W_{\mu_{N,D}}(P_tf) \Le \exp\bigg(-\frac{t}{1+\frac{3(M_2^2/a^2+1)}{2\lambda_1}}\bigg) W_{\mu_{N,D}}(f) \Le \exp(-2\lambda_2 t) W_{\mu_{N,D}}(f),~~~~
	\forall t\Ge0,
\end{equation*}
which completes the proof.
\end{proof}
\red{The convergence rate $\lambda_2$ does not depend on the number of modes $N$ or the discretization size $D$. Nevertheless, Theorem~\ref{theorem: under ergodicity} requires the special Assumption (iii), i.e., $N\Le D$.}

\subsection{Extension of the the standard PIMD}
\red{In Theorem~\ref{theorem: under ergodicity}, we have proved the uniform-in-$N$ ergodicity of the underdamped Matsubara mode PIMD \eqref{under} under Assumptions (i)(ii)(iii). In particular, Assumption (iii) is satisfied when the discretization size $D = N$ or $D = \infty$. In these two cases, the Boltzmann distribution of the continuous loop is given by
\begin{equation*}
\begin{aligned}
(D=N)~:~  \mu_{N,N}(\xi) & \propto \exp\bigg\{
	-\frac12\sum_{k=0}^{N-1}
	(\omega_k^2+a) (|\xi_k|^2+|\eta_k|^2) - \beta_N\sum_{j=0}^{N-1} V^a
	\bigg(\sum_{k=0}^{N-1} \xi_k c_k(j\beta_D)\bigg)
	\bigg\}, \\
(D=\infty)~:~~~~ \mu_{N}(\xi) & \propto \exp\bigg\{
	-\frac12\sum_{k=0}^{N-1}
	(\omega_k^2+a) (|\xi_k|^2+|\eta_k|^2) - \int_0^\beta V^a
	\bigg(\sum_{k=0}^{N-1} \xi_k c_k(\tau)\bigg)\d\tau
	\bigg\}.
\end{aligned}
\end{equation*}
Theorem~\ref{theorem: under ergodicity} directly shows the Matsubara mode PIMD for sampling $\pi_{N,N}(\xi)$ and $\pi_N(\xi)$ has the uniform-in-$N$ convergence rate $\lambda_2 = \frac{a^2}{3M_2^2+5a^2} \exp(-4\beta M_1)$.

Given the ergodicity results of the Matsubara mode PIMD, we can conveniently extend the uniform-in-$N$ ergodicity to the standard PIMD. Let $N$ be an odd integer, we employ the $N$ normal mode coordinates $\{\xi_k\}_{k=0}^{N-1}$ to represent the $N$ beads of the ring polymer (see Appendix~\ref{appendix: normal}). The Boltzmann distribution of the $N$ beads is given in \eqref{under N}:
\begin{equation*}
\mu_{N,N}^{\std}(\xi,\eta) \propto \exp\bigg\{
-\frac12\sum_{k=0}^{N-1} (\omega_{k,N}^2+a) (|\xi_k|^2+|\eta_k|^2) - \beta_N
\sum_{j=0}^{N-1}V^a\bigg(\sum_{k=0}^{N-1} \xi_k c_k(j\beta_N)\bigg)
\bigg\},
\end{equation*}
and the standard PIMD for sampling $\mu_{N,N}^{\std}(\xi,\eta)$ is given by
\begin{equation}
\left\{
\begin{aligned}
\dot \xi_k & = \eta_k, \\
\dot \eta_k & = -\xi_k - \frac{
\beta_N}{\omega_{k,N}^2+a} \sum_{j=0}^{N-1} \nabla V^a(x_N(j\beta_N))c_k(j\beta_N)
 -\eta_k + \sqrt{\frac{2}{\omega_{k,N}^2+a}}\dot B_k,
\end{aligned}
\right.
\label{under 1}
\end{equation}
where we choose the damping rate $\gamma=1$. Similar to the quantity \eqref{entropy_2} in the Matsubara mode PIMD, we define the entropy-like quantity of $f(\xi,\eta)$ in $\mathbb R^{2dN}$:
\begin{equation*}
	W_{\mu_{N,N}^{\std}}(f) := \bigg(\frac{M_2^2}{a^2}+1\bigg) \Ent_{\mu_{N,N}^{\std}}(f) + \sum_{k=0}^{N-1} \frac1{\omega_{k,N}^2+a}
	\int_{\mathbb R^{2dN}} \frac{|\nabla_{\eta_k}f - \nabla_{\xi_k}f|^2 + |\nabla_{\eta_k}f|^2}{f}\d\mu_{N,N}^{\std}.
\end{equation*}
Utilizing the same approaches with Theorem~\ref{theorem: under ergodicity}, we can prove the uniform-in-$N$ ergodicity of the standard PIMD \eqref{under 1}.
\begin{theorem}
\label{theorem: standard ergodicity}
Let $N$ be an odd integer. Under Assumptions (i)(ii), let $(P_t)_{t\Ge0}$ be the Markov semigroup of the standard PIMD \eqref{under 1}, then for any positive smooth function $f(\xi,\eta)$ in $\mathbb R^{2dN}$,
\begin{equation*}
	W_{\mu_{N,N}^{\std}}(P_tf) \Le \exp(-2\lambda_2 t) W_{\mu_{N,N}^{\std}}(f),~~~~
	\forall t\Ge0,
\end{equation*}
where the convergence rate $\lambda_2 = \frac{a^2}{3M_2^2+5a^2} \exp(-4\beta M_1)$.
\end{theorem}
\noindent
From Theorem~\ref{theorem: standard ergodicity}, we conclude that the standard PIMD has uniform-in-$N$ ergodicity.
\begin{remark}
The uniform-in-$N$ ergodicity of the standard PIMD \eqref{under 1} should also hold true when $N$ is an even integer, and it requires a slight modification on the normal mode transform \eqref{normal transform}.
\end{remark}

We note that Theorem~3.5 of \cite{pHMC} also provides a result of the dimension-free ergodicity of the PIMD, and the major difference is that \cite{pHMC} studies the preconditioned Hamiltonian Monte Carlo (pHMC) rather than the preconditioned Langevin dynamics, which is a more popular thermostat used in the PIMD \cite{simple}. Also, it is required in (3.10) that the duration parameter in the pHMC needs to be small enough.}
\section{Numerical tests}
\label{section: numerical result}
\subsection{Setup of parameters and the time average estimator}
\red{For several examples of quantum systems, we compute the quantum thermal average employing the Matsubara mode PIMD \eqref{under} and the standard PIMD \eqref{under 1}. For simplicity, we choose the preconditioning parameter $a=1$ and the damping rate $\gamma = 1$. Also, the discretization size $D$ equal to the number of modes $N$ is an odd integer.

In the Matsubara mode PIMD, the quantum thermal average $\avg{O(\hat q)}_\beta$ is approximated by th statistical average $\avg{O(\hat q)}_{\beta,N,N}$:
\begin{equation*}
	\avg{O(\hat q)}_\beta \approx
	\avg{O(\hat q)}_{\beta,N,N} :=
	\int_{\mathbb R^{dN}}
	\bigg[
	\frac1N\sum_{j=0}^{N-1}
	O\bigg(\sum_{k=0}^{N-1} \xi_k c_k(j\beta_N)\bigg)
	\bigg]\pi_{N,N}(\xi)\d\xi,
\end{equation*}
where $\pi_{N,N}(\xi)$ is the Boltzmann distribution of the continuous loop defined in \eqref{pi N}.
Similarly, in the standard PIMD we have the approximation
\begin{equation*}
	\avg{O(\hat q)}_\beta \approx
	\avg{O(\hat q)}_{\beta,N,N}^{\std} :=
	\int_{\mathbb R^{dN}}
	\bigg[
	\frac1N\sum_{j=0}^{N-1}
	O\bigg(\sum_{k=0}^{N-1} \xi_k c_k(j\beta_N)\bigg)
	\bigg]\pi_{N,N}^{\std}(\xi)\d\xi,
\end{equation*}
Finally, let $\xi_k(t)$ and $\eta_k(t)$ be the solution to the Matsubara mode PIMD \eqref{under} or the standard PIMD \eqref{under 1}, then the quantum thermal average is computed from the time average
\begin{equation*}
	A(\beta,N,T) := \frac1T\int_0^T \bigg[
		\frac1N\sum_{j=0}^{N-1}
		O\bigg(\sum_{k=0}^{N-1} \xi_k(t) c_k(j\beta_N)\bigg)
		\bigg],
\end{equation*}
where $T>0$ is the simulation time. The accuracy of the simulation can thus be characterized by the time average error $A(\beta,N,T) - \avg{O(\hat q)}_\beta$, which depends on the number of modes $N$ and the simulation time $T$.

For the discretization of the Langevin dynamics, we employ the standard BAOAB integrator \cite{md}, which is widely applied in the molecular dynamics \cite{md,md_1}.}
\subsection{1D model potential}
Let the potential function $V(q)$ and the observable function $O(q)$ be given by
\begin{equation}
	V(q) = \frac12q^2 + q\cos q,~~~~ O(q) = \sin\bigg(\frac{\pi}2q\bigg),~~~~q\in\mathbb R^1.
	\label{model: 1}
\end{equation}
The exact value of the quantum thermal average $\avg{O(\hat q)}_\beta$ is computed from the spectral method with the Gauss--Hermite quadrature.
In the numerical tests, fix the time step $\Delta t = 1/16$ and the simulation time $T = 5\times10^6$.
\subsubsection{Convergence of the time average}
We plot the time average error in computing the quantum thermal average $\avg{O(\hat q)}_\beta$ in Figure~\ref{figure: 1}. The inverse temperature $\beta = 1,2,4,8$, and the number of modes $N = 9,17,33,65,129$. The left and right columns show the result of the Matsubara mode PIMD and the standard PIMD, respectively.
\begin{center}
\includegraphics[width=0.48\textwidth]{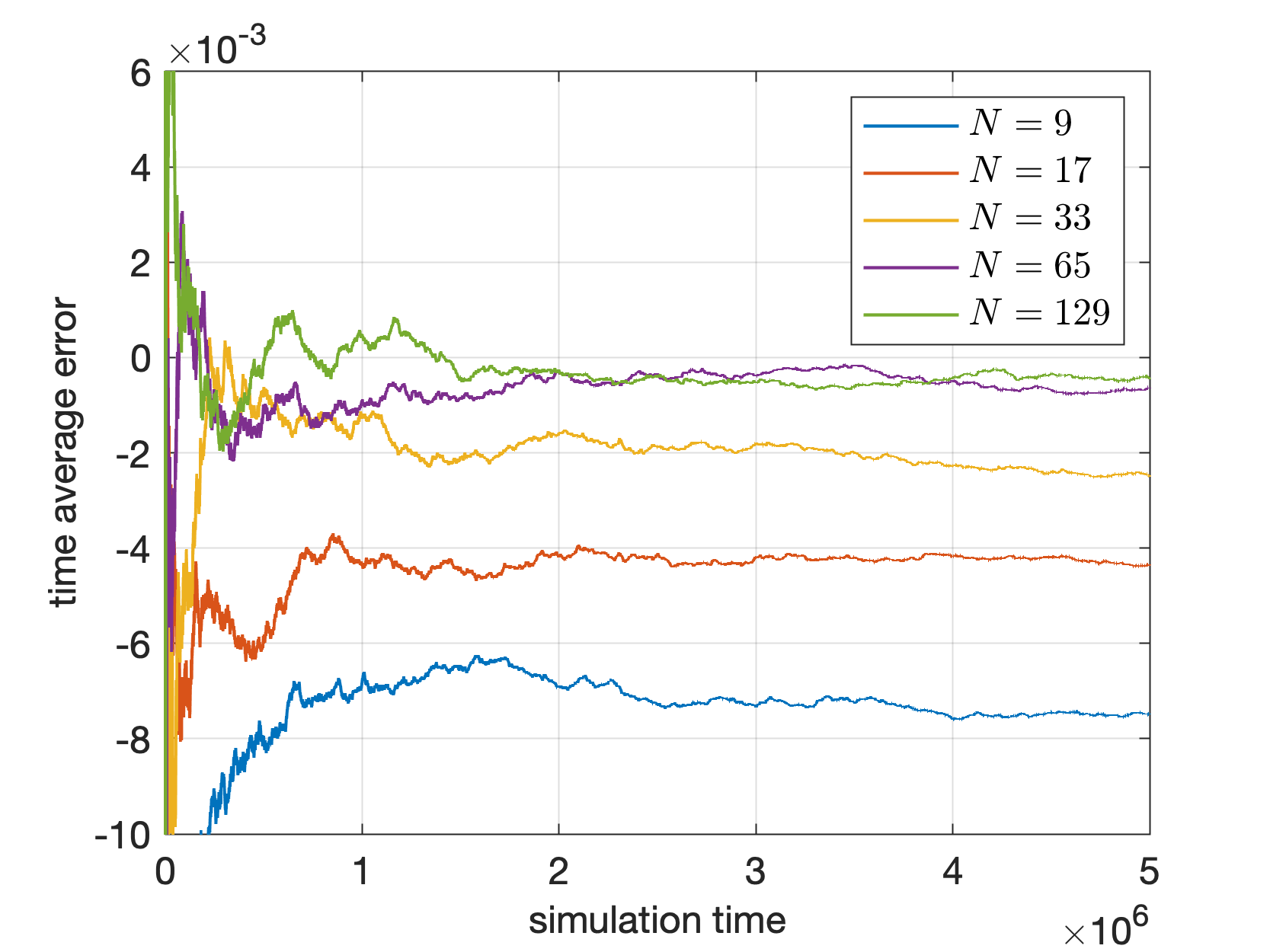}
\includegraphics[width=0.48\textwidth]{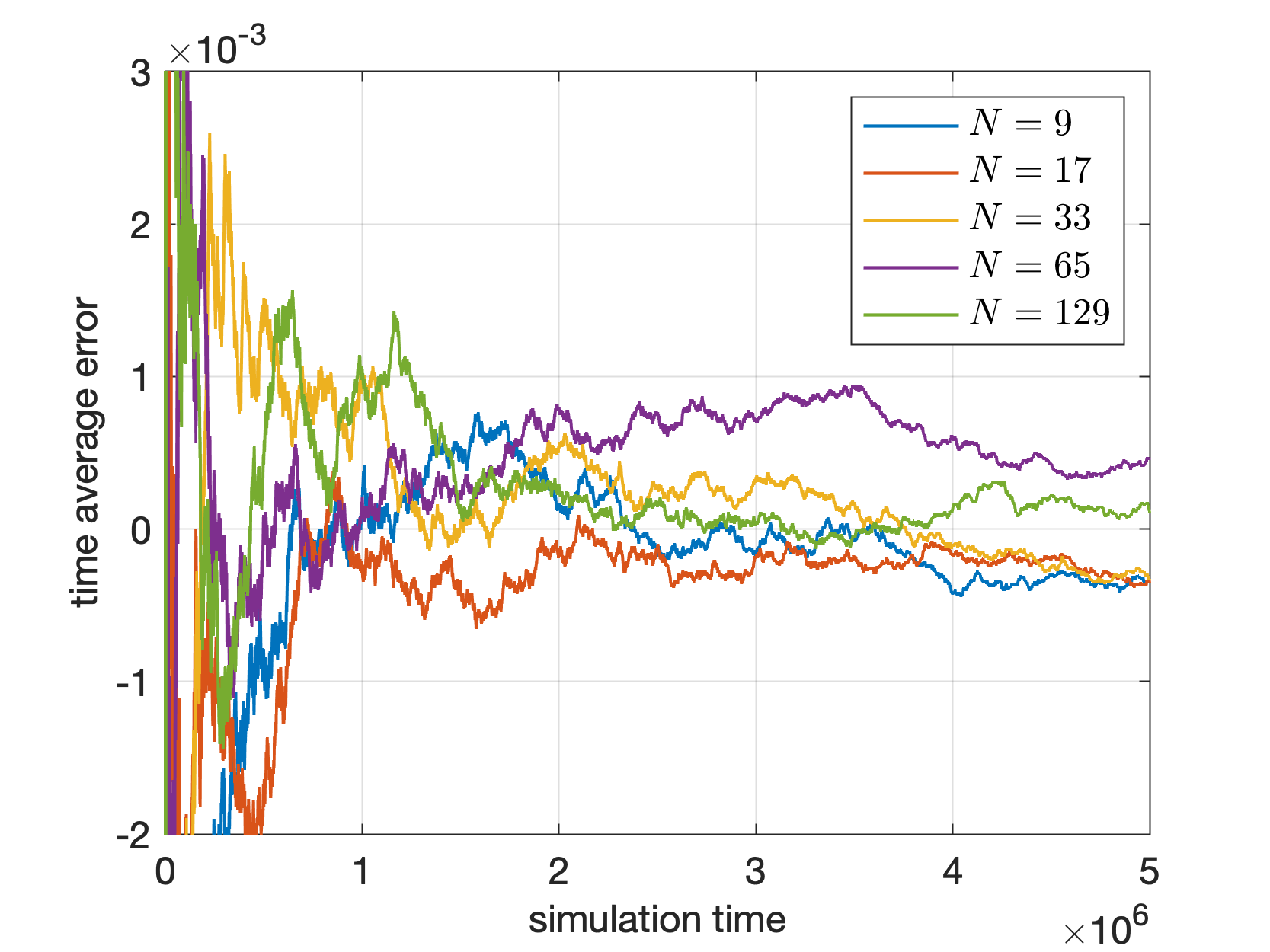} \\
\includegraphics[width=0.48\textwidth]{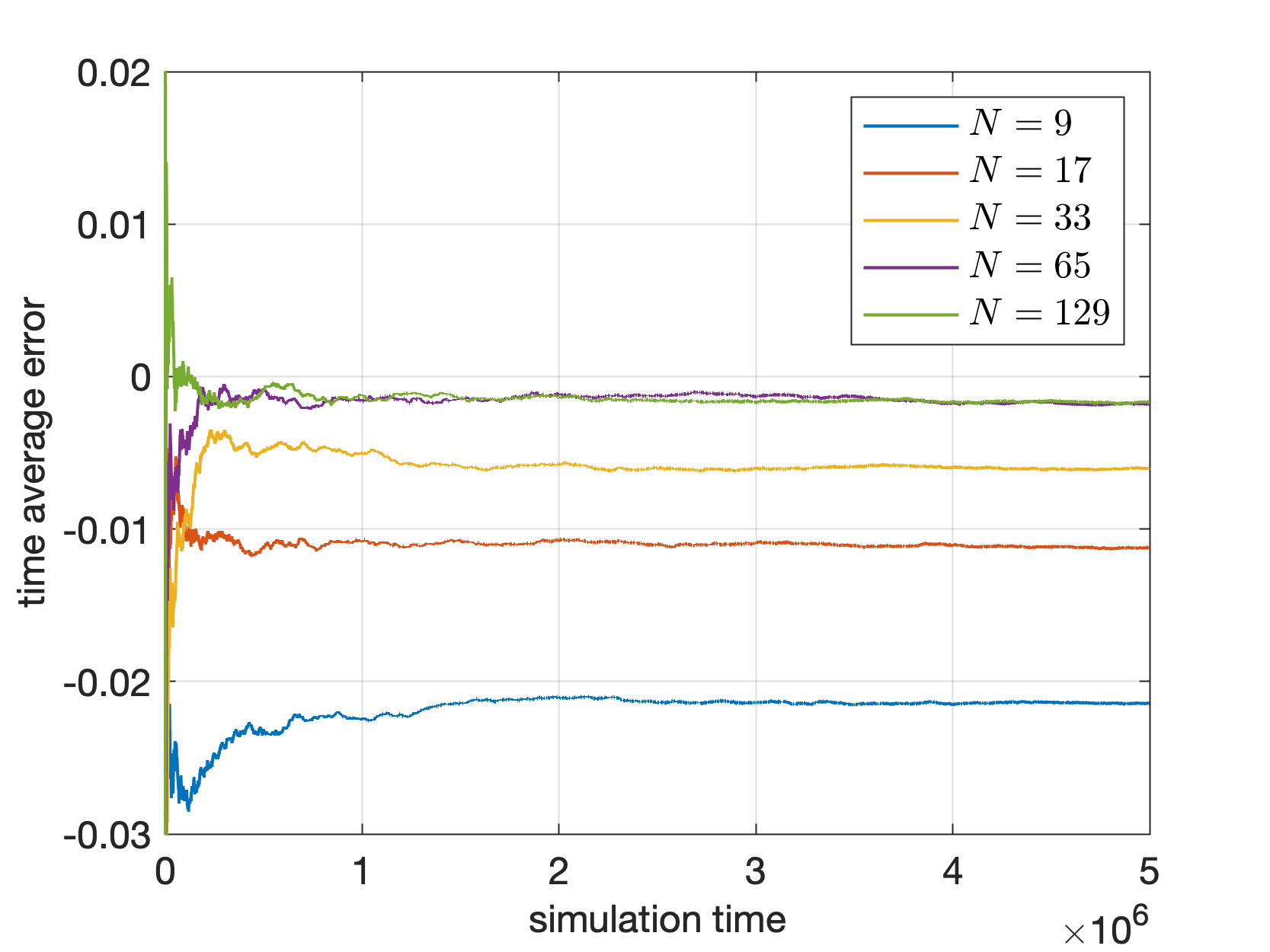}
\includegraphics[width=0.48\textwidth]{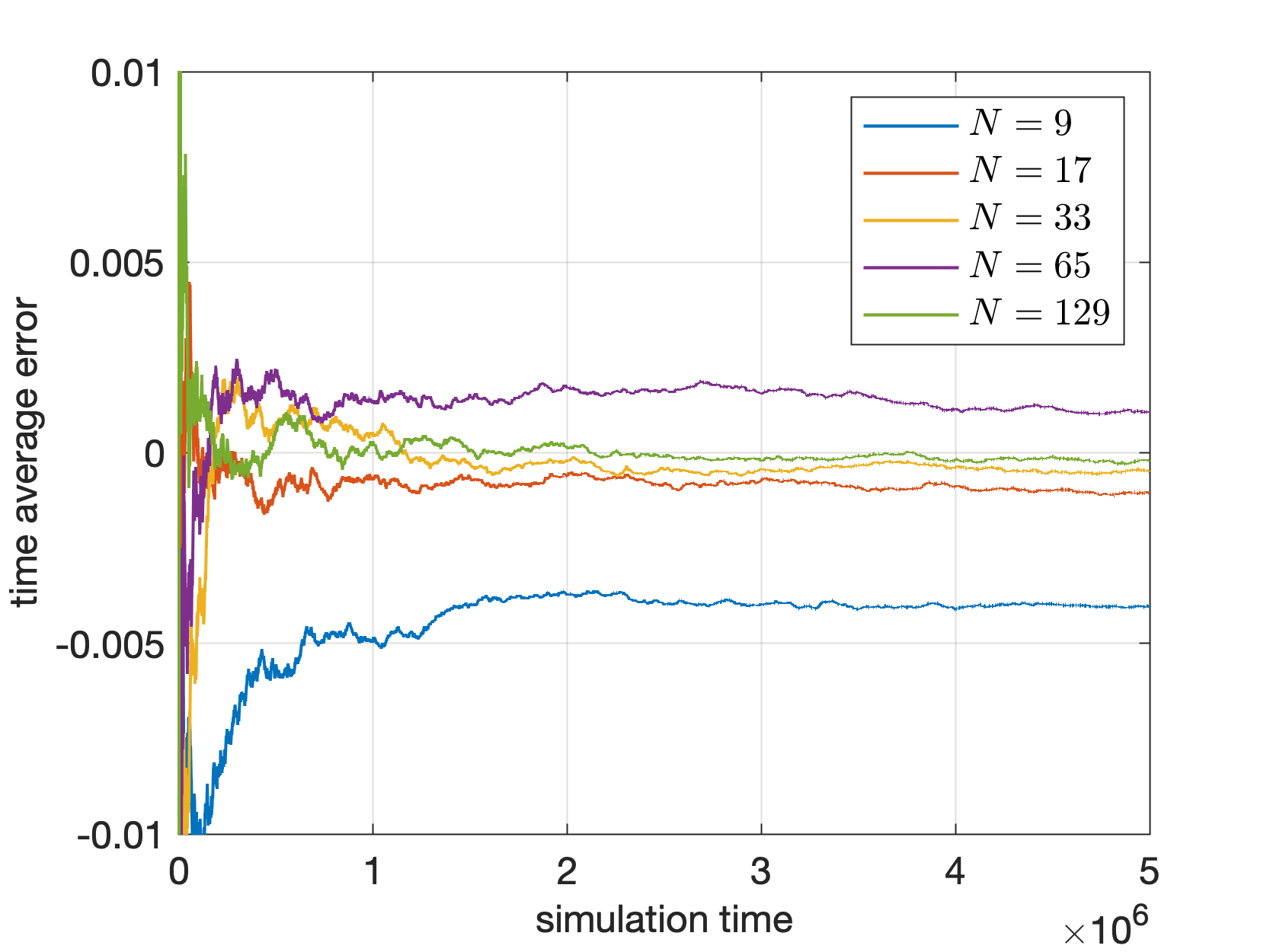} \\
\includegraphics[width=0.48\textwidth]{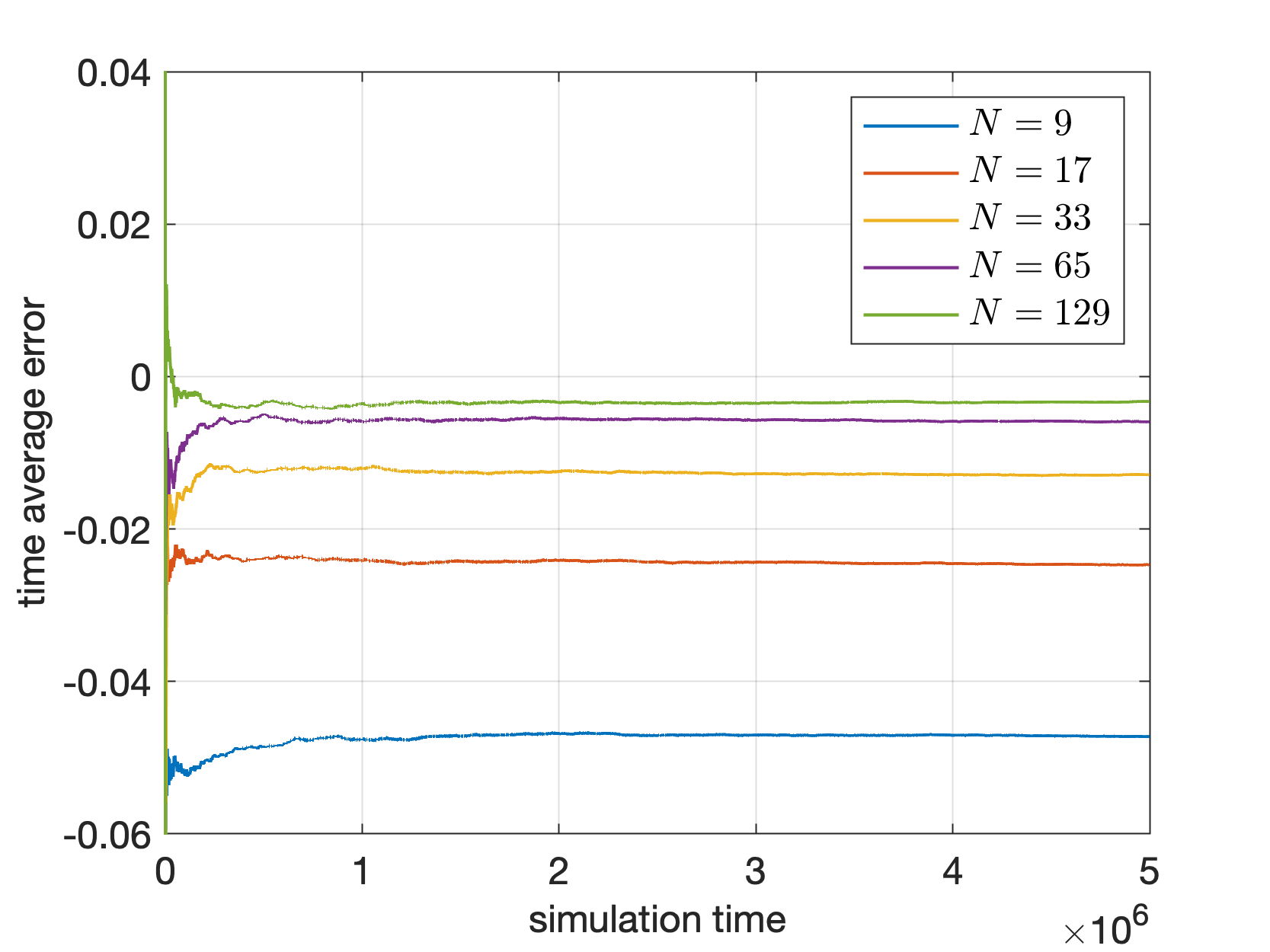}
\includegraphics[width=0.48\textwidth]{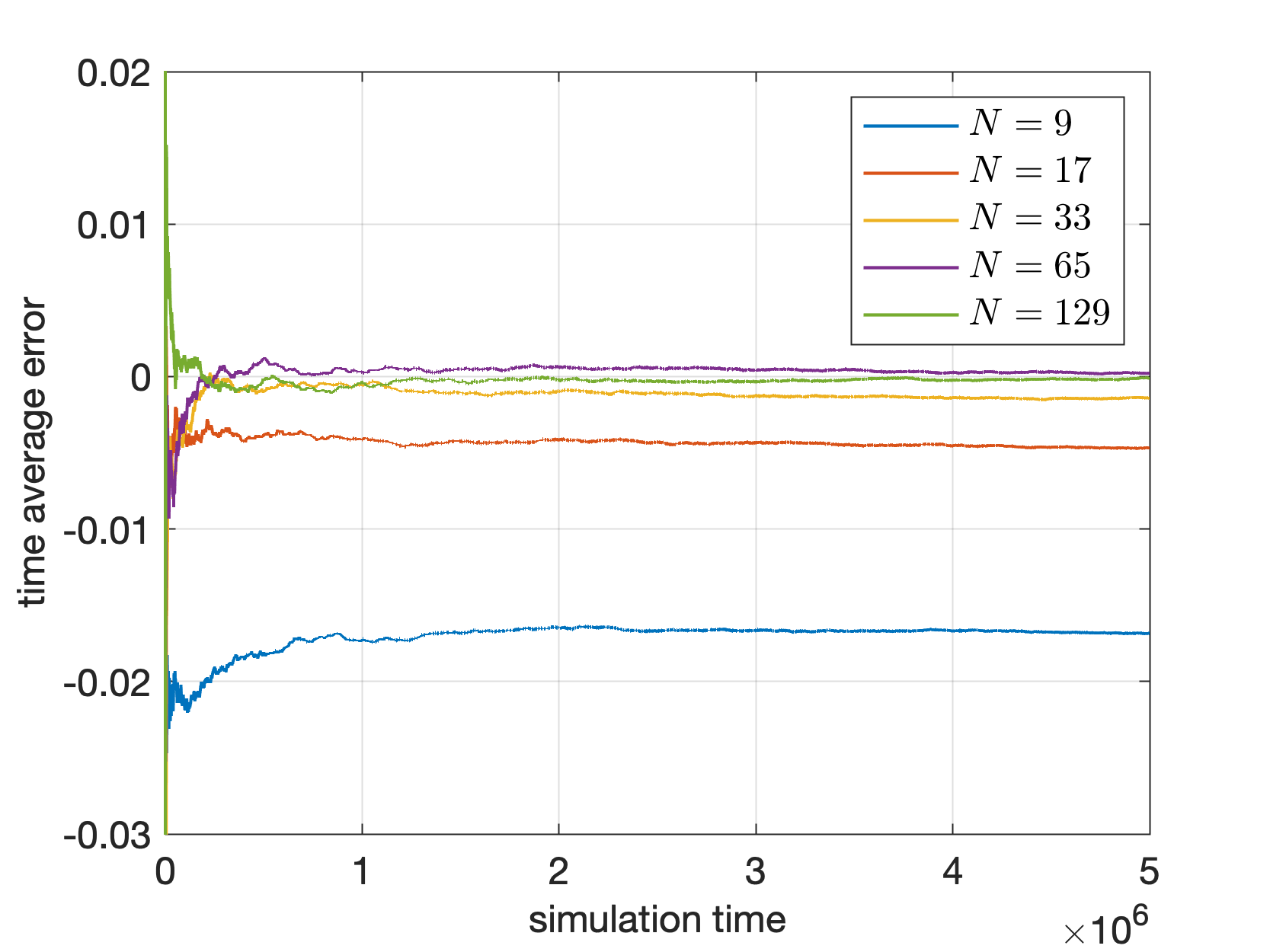}
\includegraphics[width=0.48\textwidth]{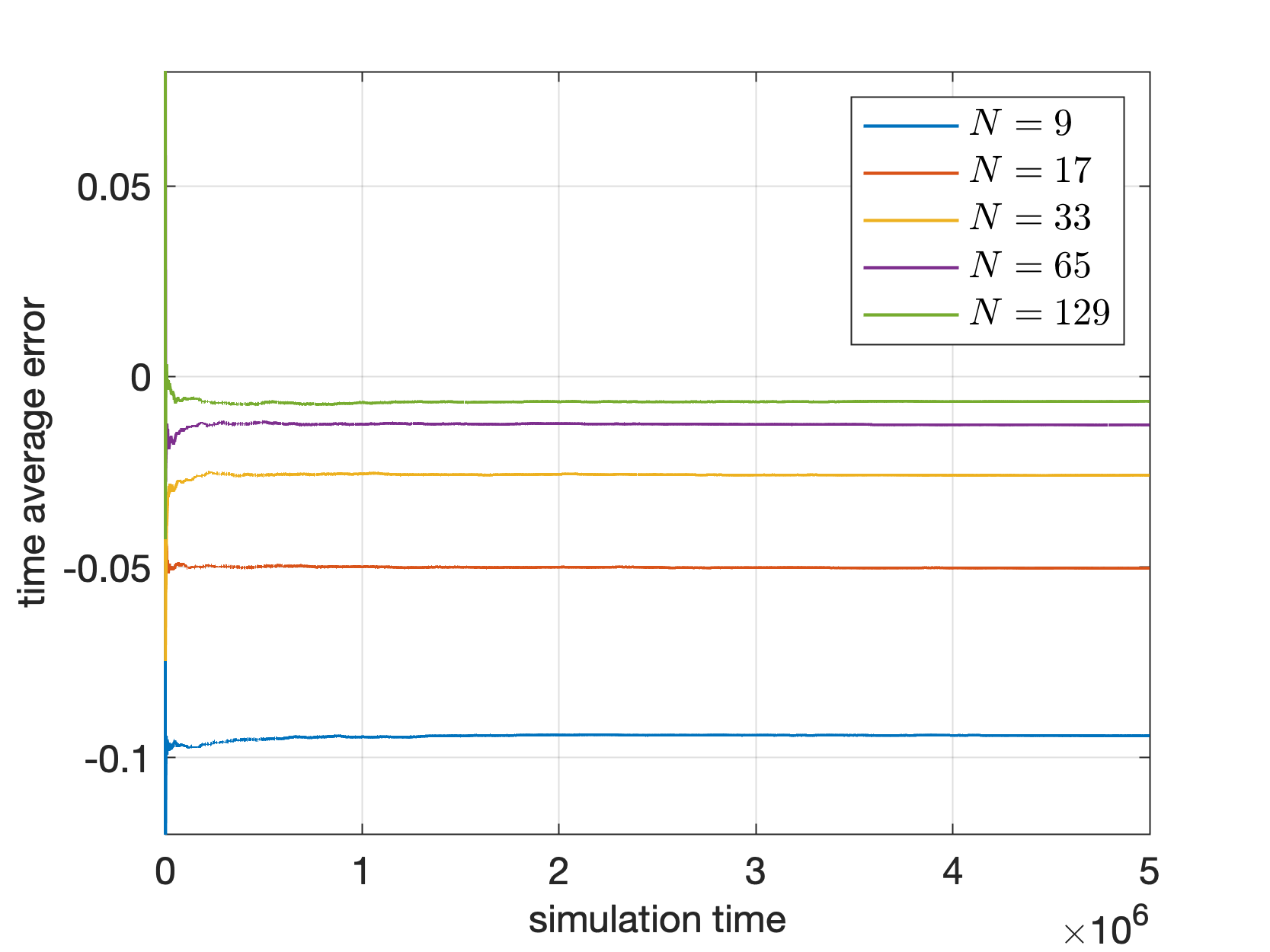}
\includegraphics[width=0.48\textwidth]{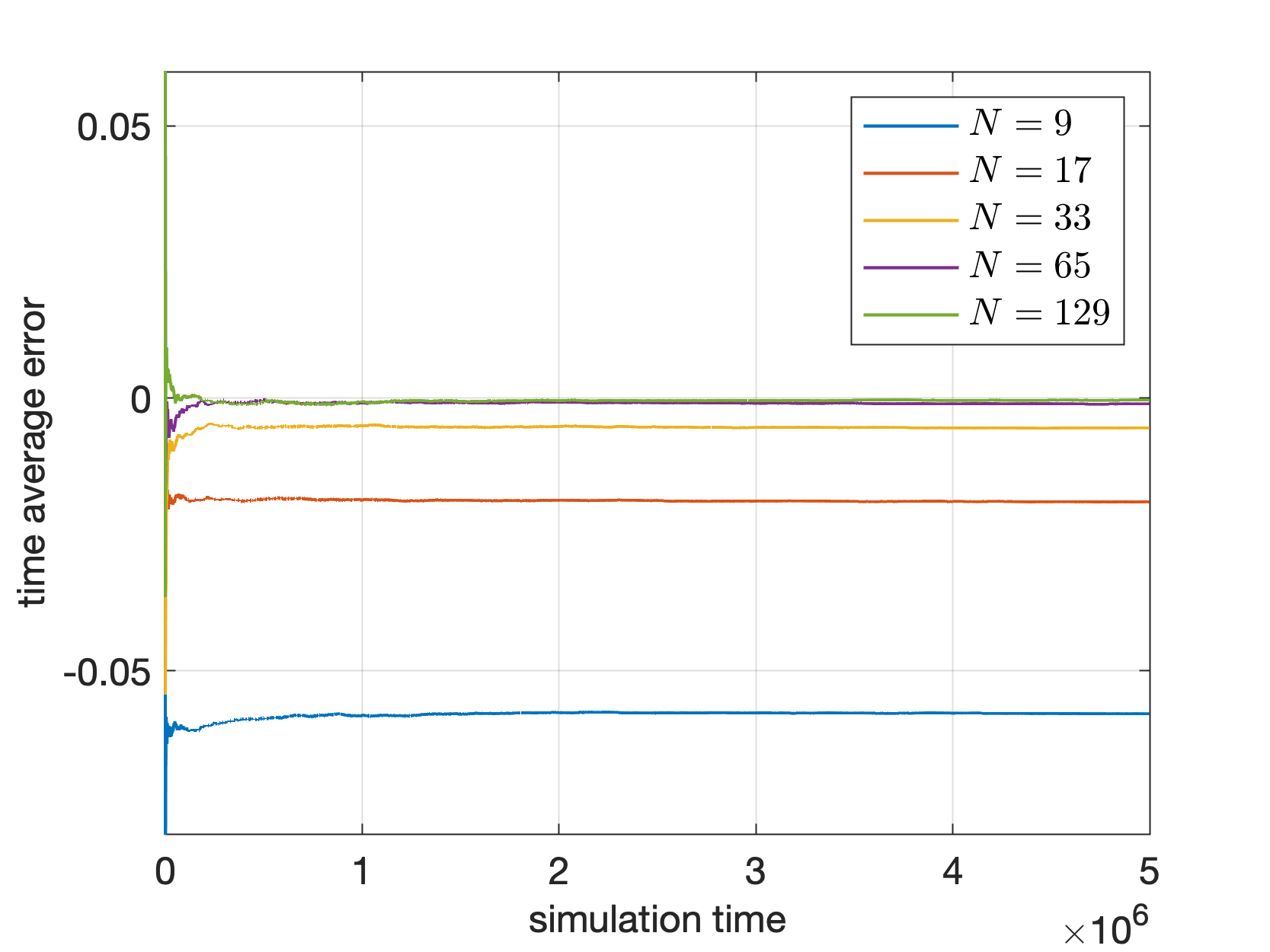}
\captionof{figure}{Time average error in computing the quantum thermal average for the 1D model potential \eqref{model: 1}. Left: Matsubara mode PIMD. Right: standard PIMD. Top to bottom: the inverse temperature $\beta = 1,2,4,8$.}
\label{figure: 1}
\end{center}
Figure~\ref{figure: 1} shows that the standard PIMD has a better accuracy than the Matsubara mode PIMD in all temperatures, and requires a smaller number of modes $N$ for convergence.

\subsubsection{Correlation function of mode coordinates}
The correlation function of the mode coordinates $\{\xi_k\}_{k=0}^{N-1}$ is computed from
\begin{equation*}
	C_k(\beta,N,\Delta T) := \frac{\avg{\xi_k(t)\xi_k(t+\Delta T)}}{\avg{\xi_k(t)\xi_k(t)}},~~~~
	k = 0,1,\cdots,N-1,
\end{equation*}
where $\avg{f(t)}:=\lim_{T\rightarrow\infty} T^{-1} \int_0^T f(t)\d t$ denotes the time average of the function $f$, and $\Delta t$ is the time interval of two instants recording the coordinate $\xi_k$. The exponential decay of $C_k(\beta,N,\Delta)$ is an important criterion of the geometric ergodicity.

We compute the correlation functions of the first five mode coordinates $\{\xi_k\}_{k=0}^4$, where the inverse temperature $\beta = 1,2,4,8$, and the number of modes $N = 9,17,33,65,129$. The correlation functions $C_k(\beta,N,\Delta T)$ are plotted as the function of $\Delta T$ in Figure~\ref{figure: 2}.
\begin{center}
\includegraphics[width=0.48\textwidth]{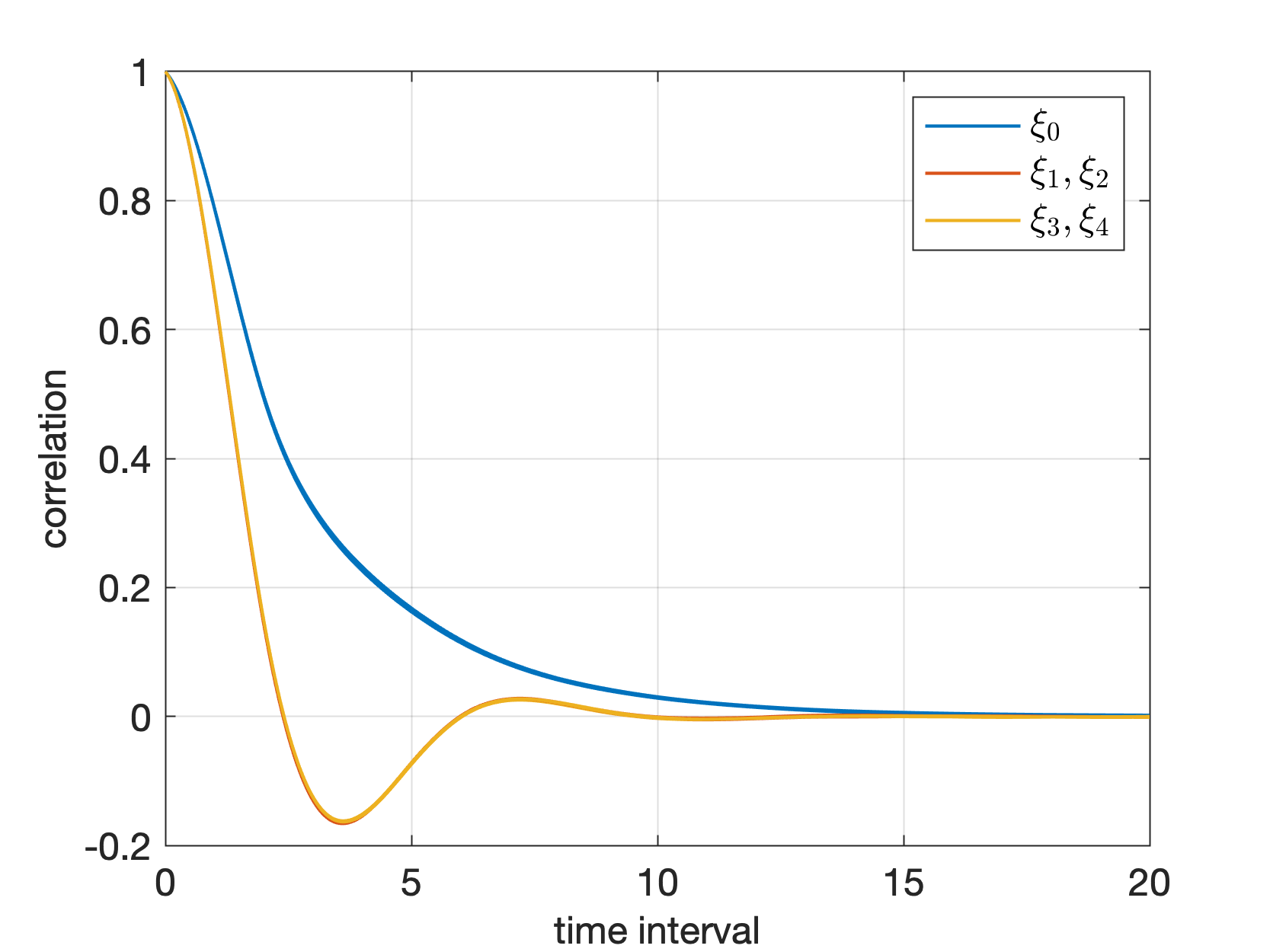}
\includegraphics[width=0.48\textwidth]{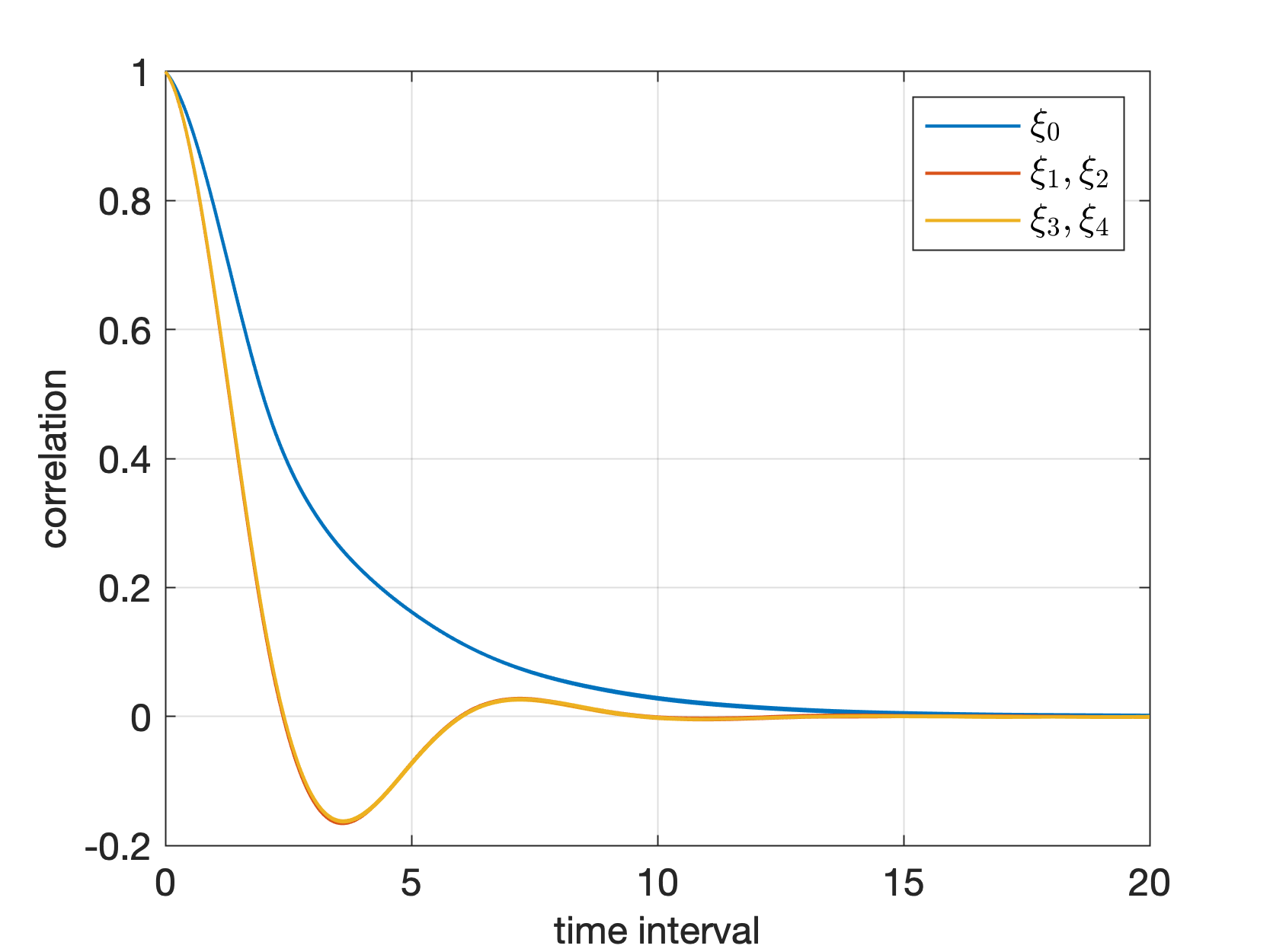} \\
\includegraphics[width=0.48\textwidth]{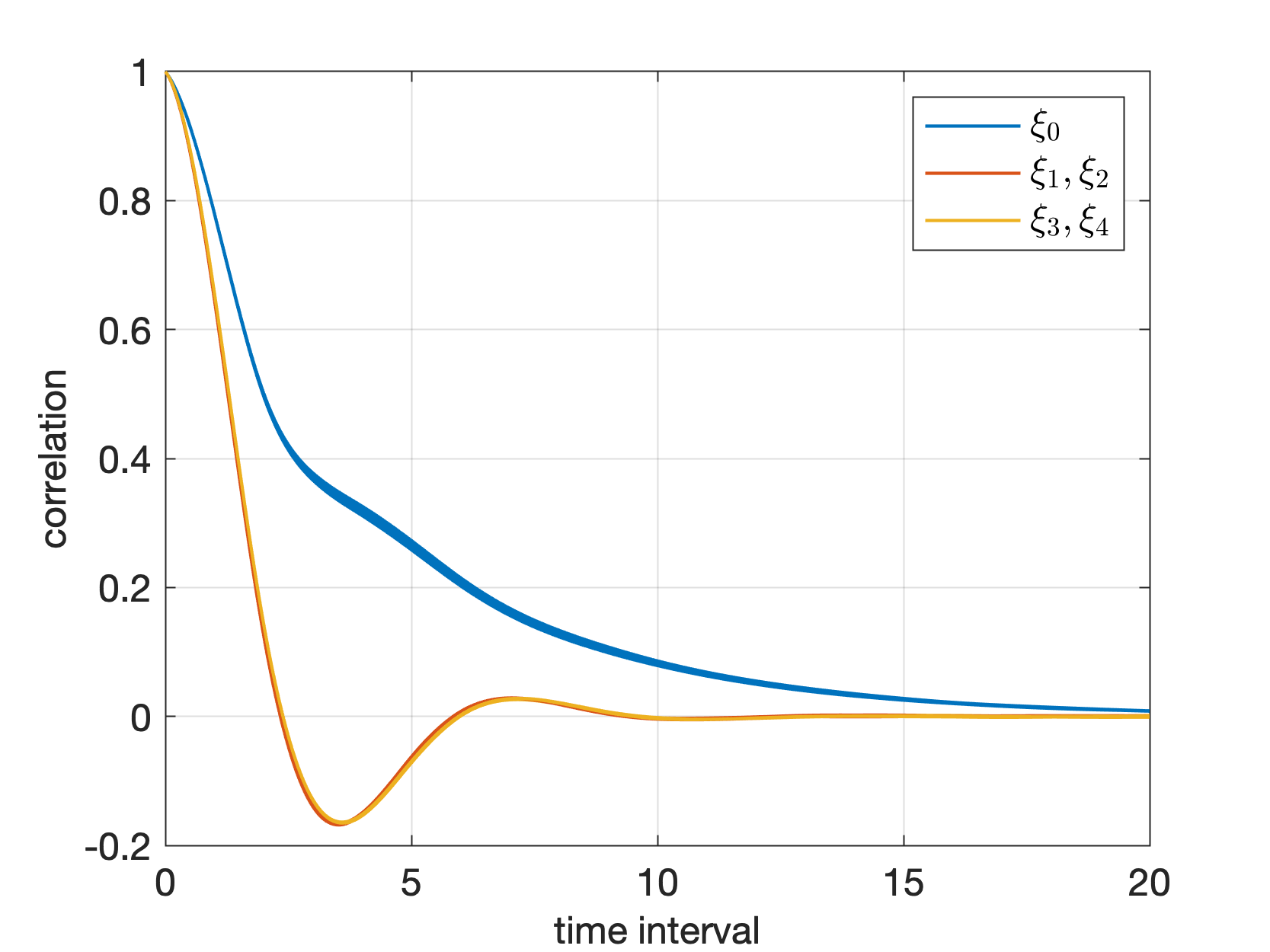}
\includegraphics[width=0.48\textwidth]{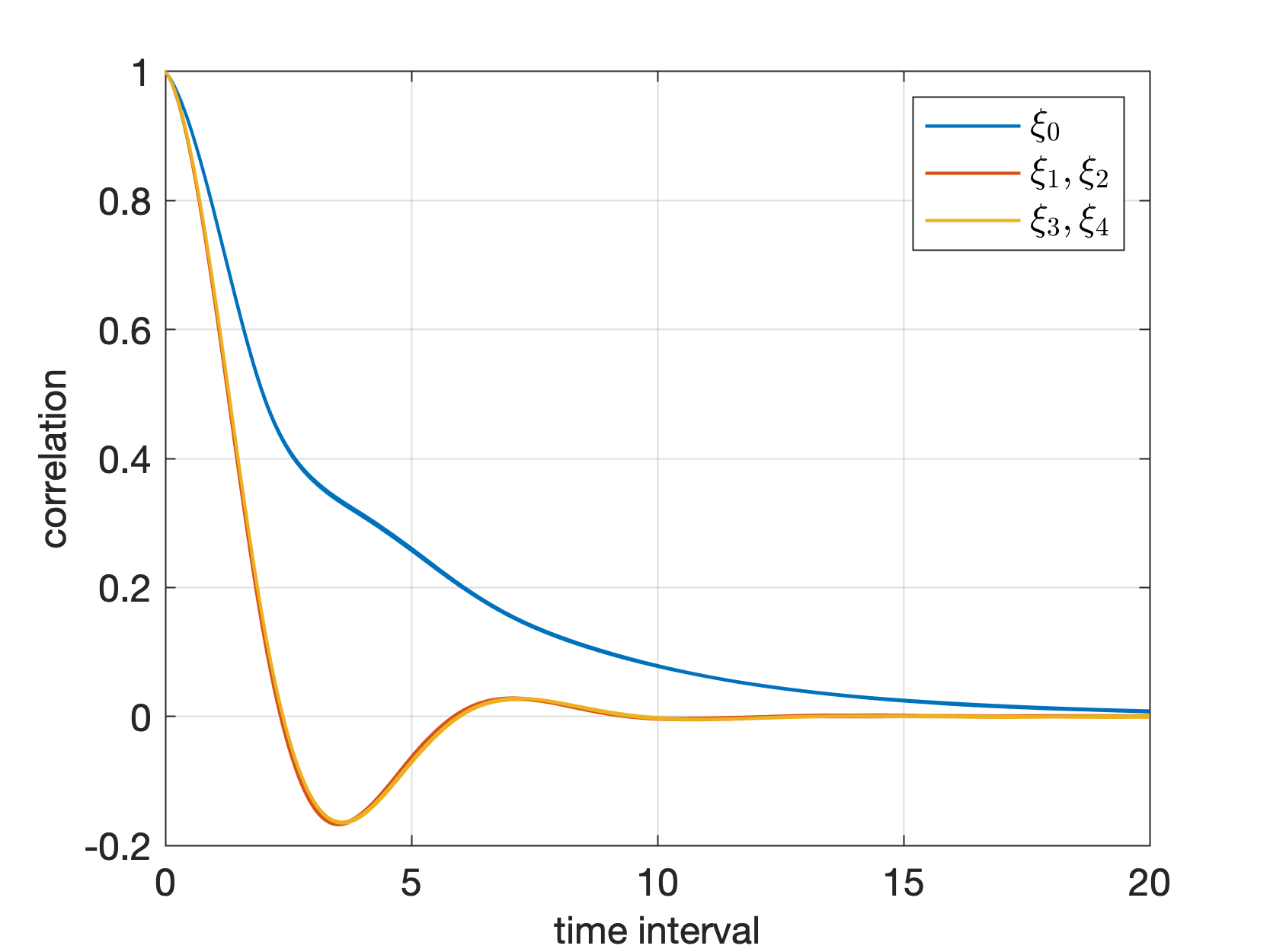} \\
\includegraphics[width=0.48\textwidth]{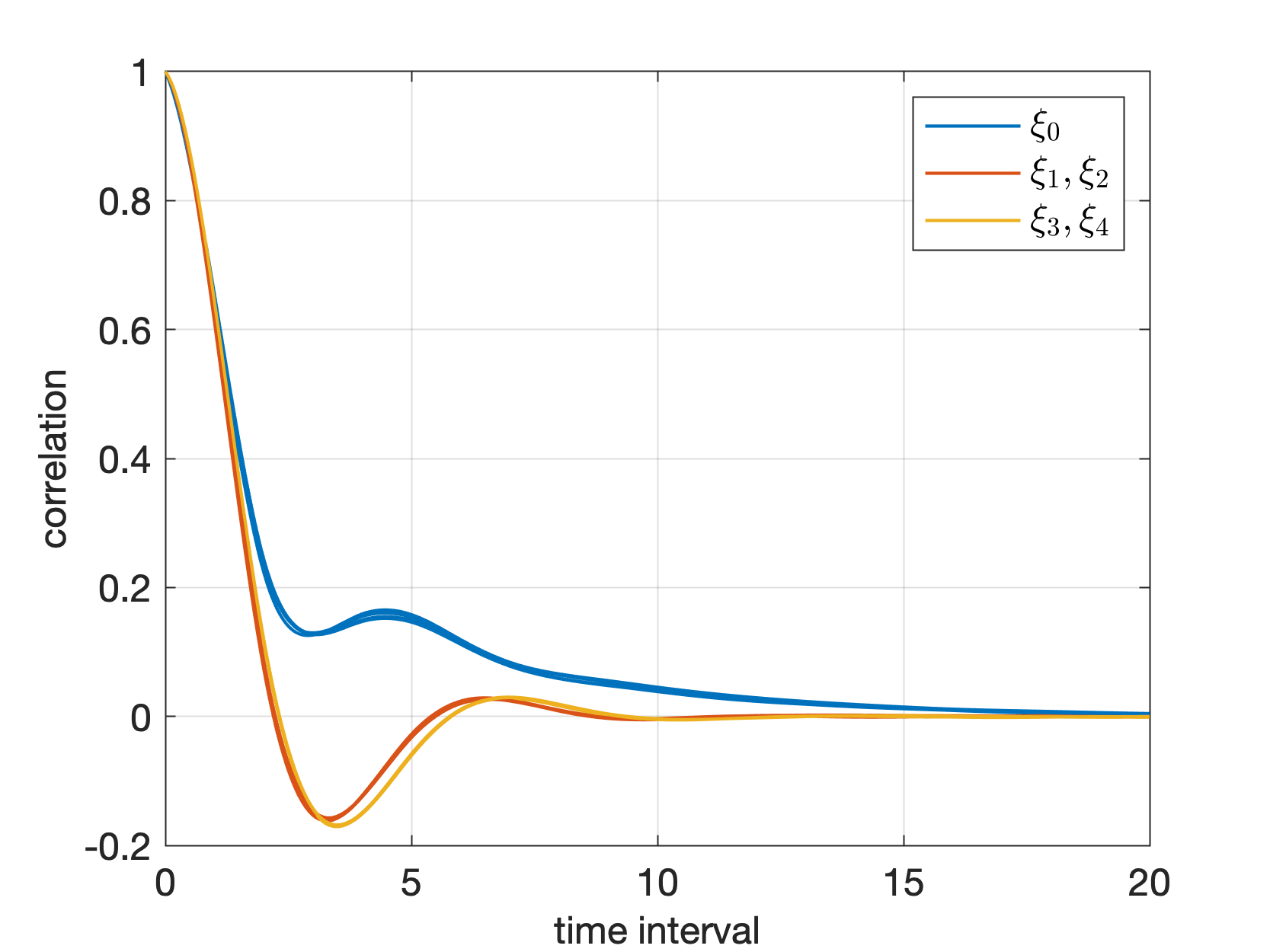}
\includegraphics[width=0.48\textwidth]{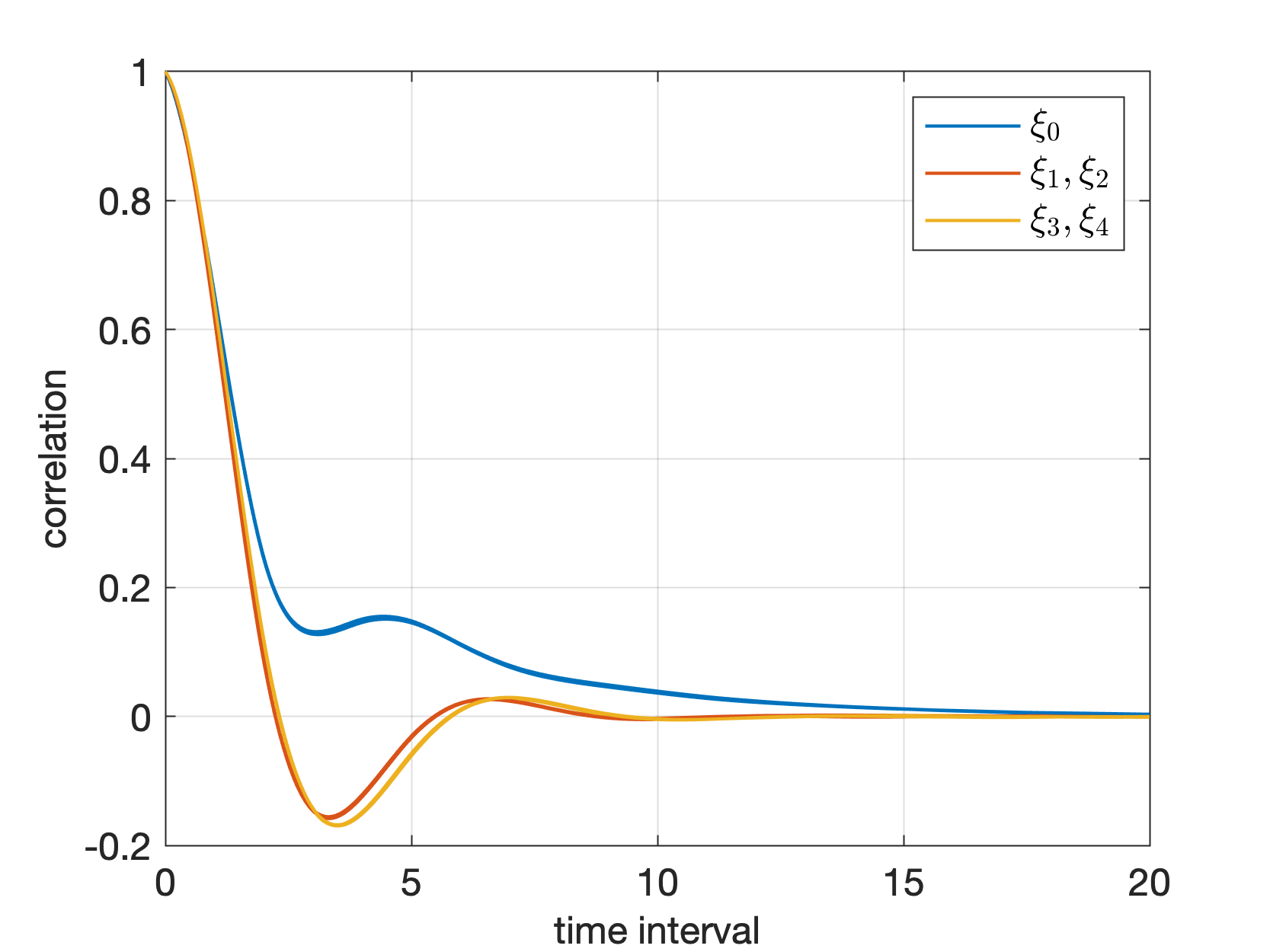} \\
\includegraphics[width=0.48\textwidth]{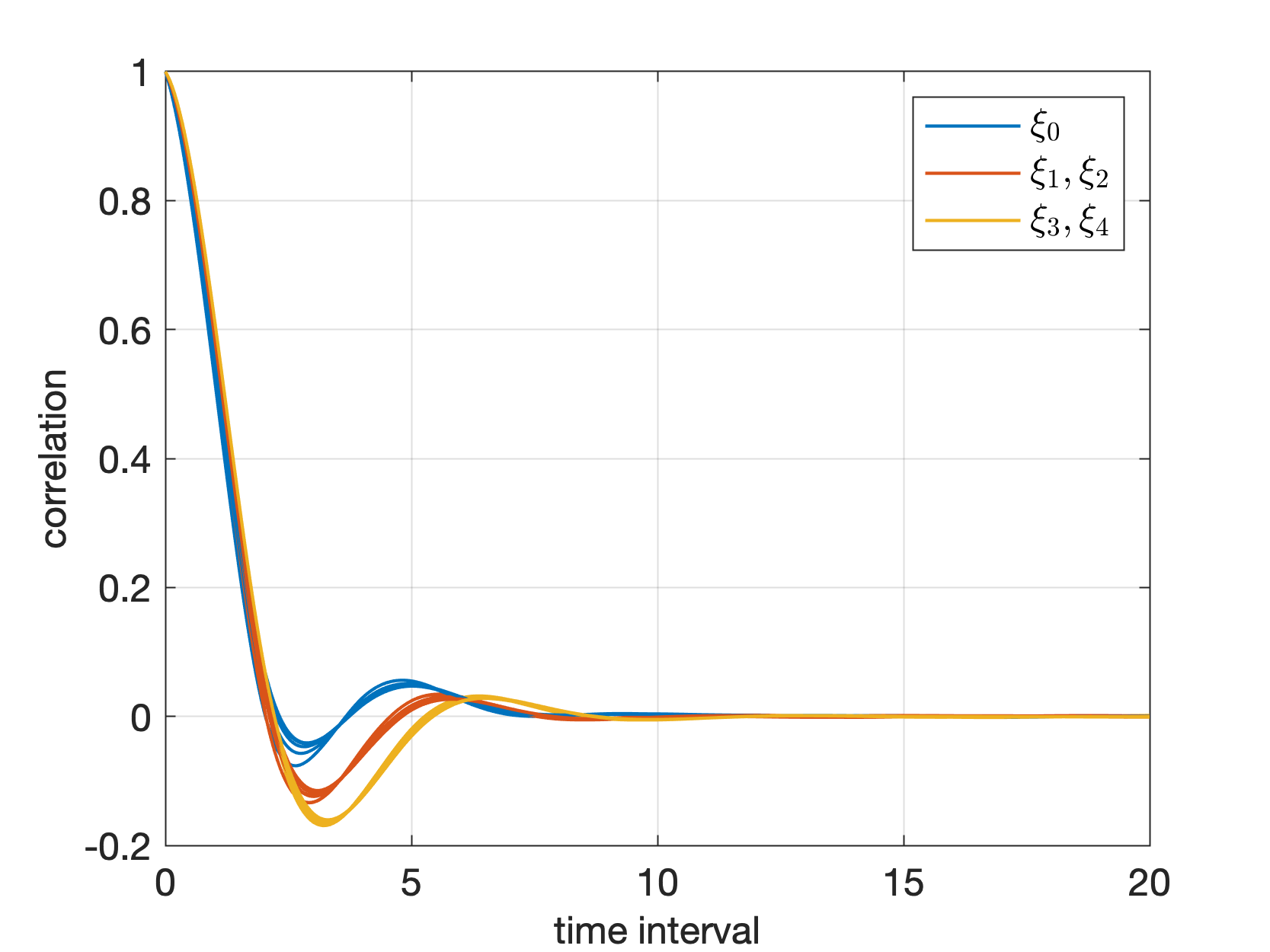}
\includegraphics[width=0.48\textwidth]{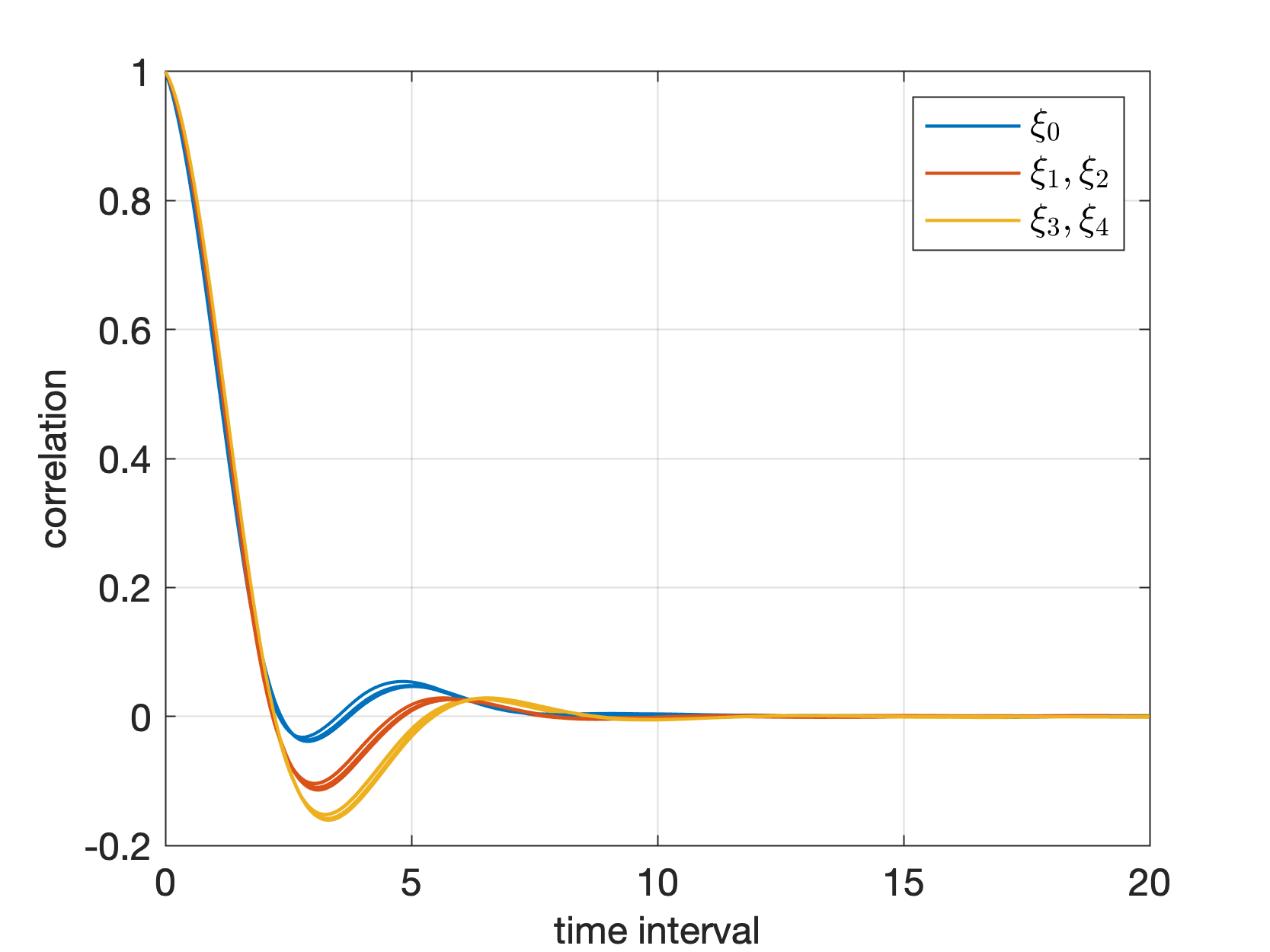}
\captionof{figure}{Correlation functions in the numerical simulation of the 1D potential \ref{model: 1}. Left: Matsubara mode PIMD. Right: standard PIMD. Top to bottom: $\beta = 1,2,4,8$. The centroid mode $\xi_0$ is colored in blue, $\xi_1,\xi_2$ are colored in red, and $\xi_3,\xi_4$ are colored in yellow.}
\label{figure: 2}
\end{center}
The coincidence of the correlation functions for various $N$ shows that both the Matsubara mode PIMD and the standard PIMD have uniform-in-$N$ ergodicity. Meanwhile, the separation of the correlation functions for various $k$ verifies the convergence rates on the different modes can be divergent, \red{and in low temperatures ($\beta$ is large), high frequency modes tend to have a longer correlation time.}
\subsection{3D spherical potential}
Consider the 3D spherical potential
\begin{equation*}
	V(q) = \frac12|q|^2+\frac1{\sqrt{|q|^2+0.2^2}},~~~~|q| = \sqrt{q_1^2+q_2^2+q_3^2},
\end{equation*}
and we aim to capture the probability distribution of $|q|$, namely, the Euclidean distance from the origin in $\mathbb R^3$. Utilizing the density operator $e^{-\beta\hat H}$, the distribution of $|q|$ can be expressed via the density function
\begin{equation*}
	\rho(r) = \frac1{\mathcal Z}\int_{\mathbb R^3} \mel{q}{e^{-\beta\hat H}}{q}\delta(|q|-r) \d q,~~~~r\Ge0,~~~~
	\mathcal Z = \Tr\big[e^{-\beta\hat H}\big].
\end{equation*}
The distribution $\rho(r)$ is able to characterize the radial observable functions. Choose inverse temperature $\beta = 4$, the time step $\Delta t = 1/32$, and the simualtion time $T = 5\times10^6$. In Figure~\ref{figure: 3}, we plot the density of $|q|$ while the number of modes $N$ varies in $3,5,9,17,33$.
\begin{center}
\includegraphics[width=0.48\textwidth]{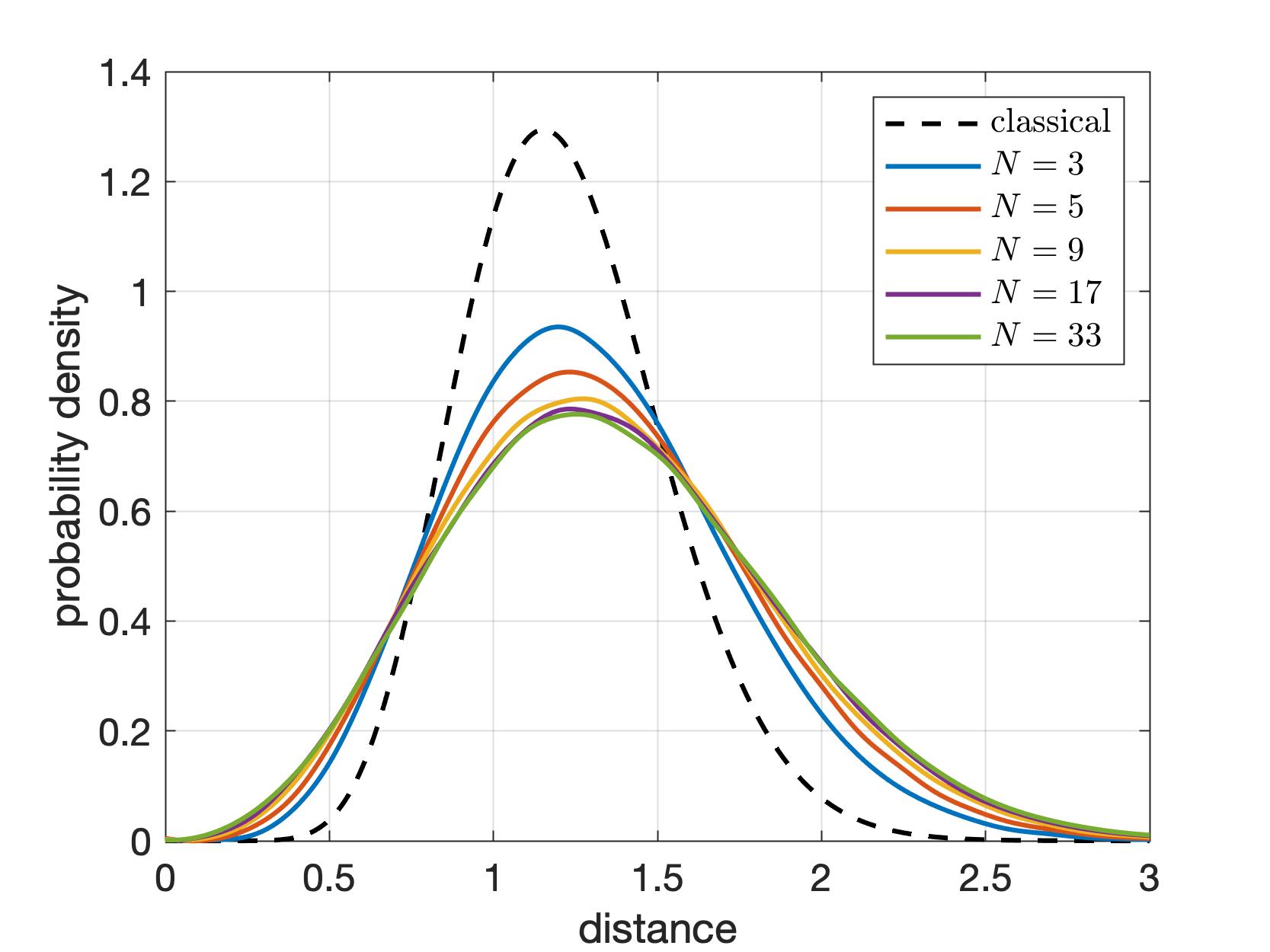}
\includegraphics[width=0.48\textwidth]{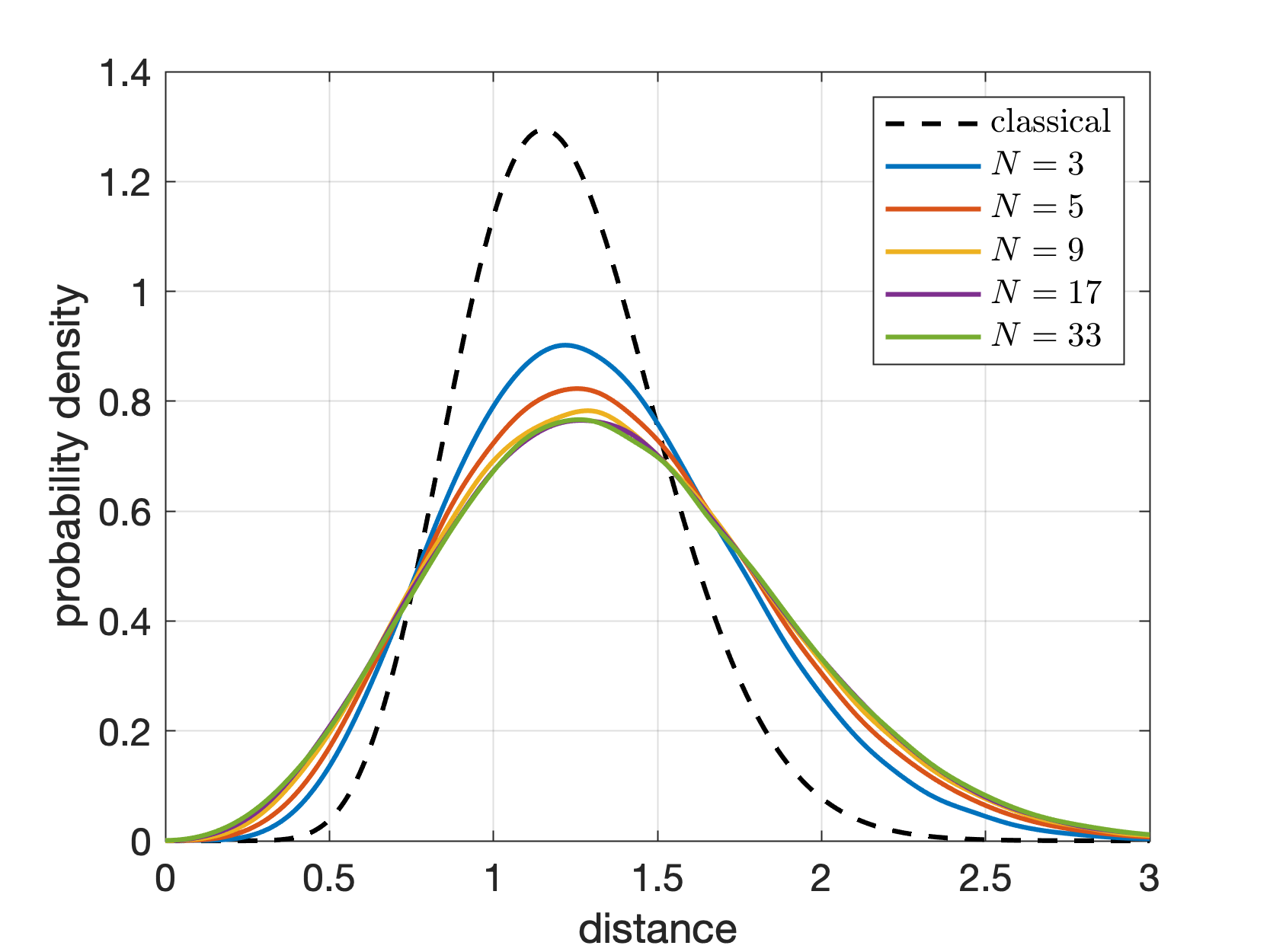}
\includegraphics[width=0.48\textwidth]{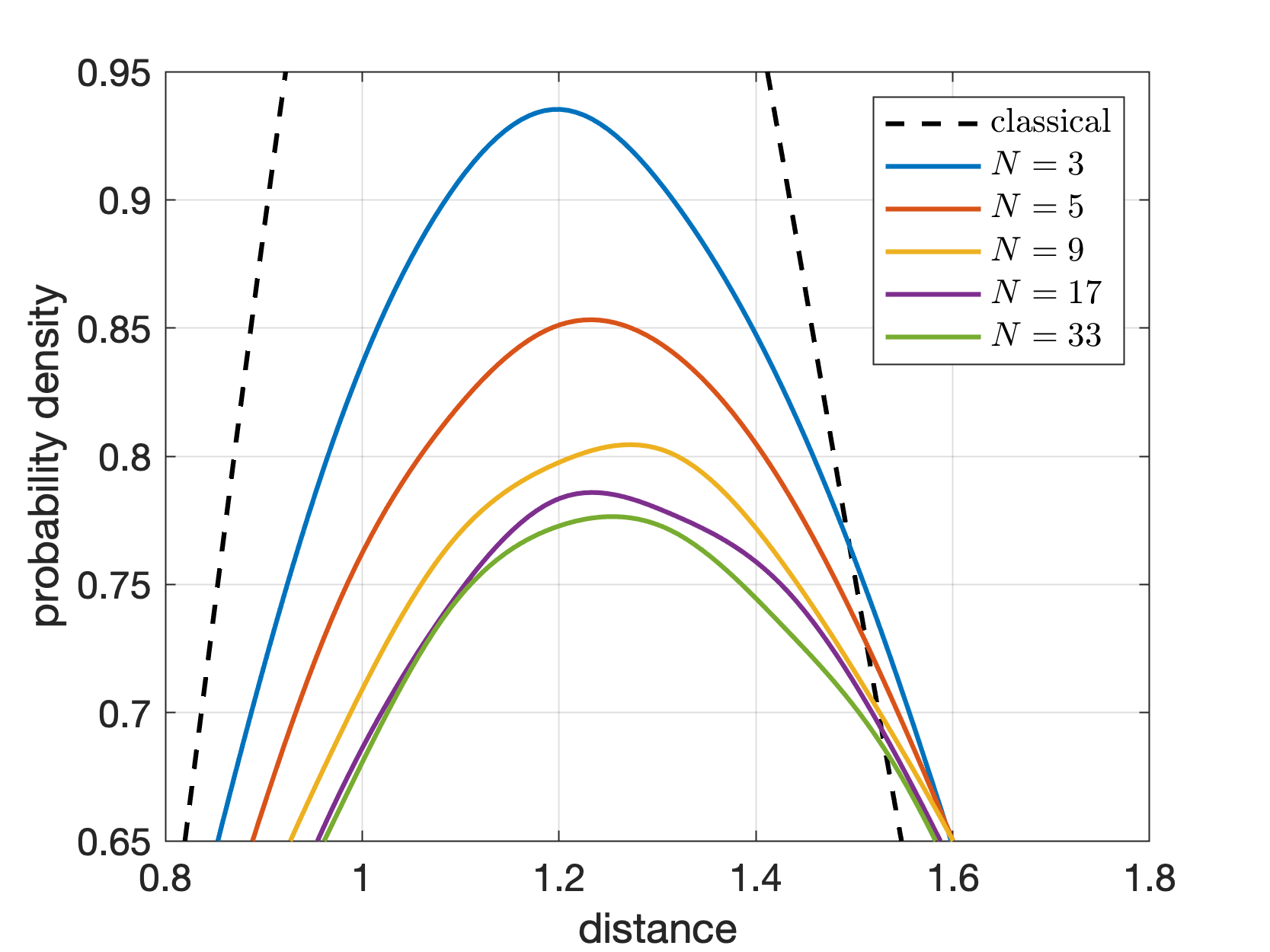}
\includegraphics[width=0.48\textwidth]{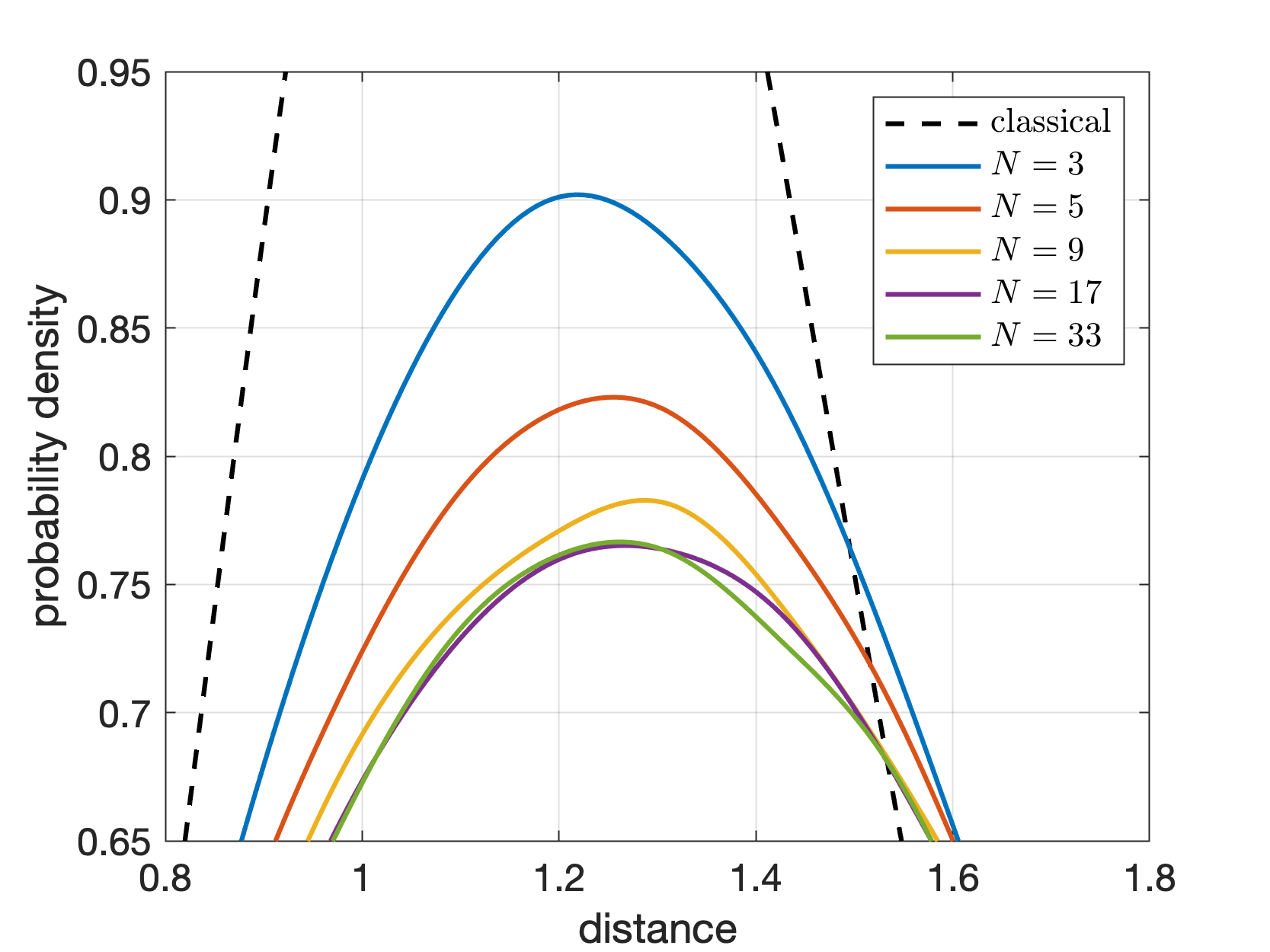}
\captionof{figure}{Probability density of $|q|$ in the simulation of the PIMD. Left: Matsubara mode PIMD. Right: standard PIMD. The top and bottom graphs use different scales.}
\label{figure: 3}
\end{center}
Figure~\ref{figure: 3} shows that as the number of modes $N$ increases, the density function shifts from the classical distribution (black dashed curve) to the quantum limit. Both the Matsubara mode PIMD and the standard PIMD correctly computes the correct density function $\rho(r)$, however the standard PIMD has better accuracy than the Matsubara mode PIMD when the number of modes $N$ is small.

\section{Conclusion}
\red{We prove the uniform-in-$N$ ergodicity of the Matsubara mode PIMD and the standard PIMD for a general potential function $V(q)$, i.e., the convergence towards the invariant distribution does not depend on the number of modes (for the Matsubara mode PIMD) or the number of beads (for the standard PIMD).}

\section*{Acknowledgments}
The work of Z. Zhou is partially supported by the National Key R\&D Program of China
(Project No. 2020YFA0712000, 2021YFA1001200), and the National Natural Science Foundation
of China (Grant No. 12031013, 12171013).

The authors would like to thank Haitao Wang (SJTU), Xin Chen (SJTU), Wei Liu (WHU) for the helpful discussions. The numerical tests are supported by High-performance Computing Platform of Peking University.

\begin{appendices}

\section{Additional proofs for Section \ref{section: convergence analysis}}
\label{appendix: proof convergence}
\begin{lemma}
\label{appendix: 1}
Suppose $V^c(q)$ is convex in $\mathbb R^d$, then for any coordinates $\xi = \{\xi_k\}_{k=0}^{N-1}$ in $\mathbb R^{dN}$,
\begin{equation*}
	\mathcal V_{N,D}^c(\xi) = \beta_D \sum_{j=0}^{D-1} V^c(x_N(j\beta_D))
\end{equation*}
is convex in $\mathbb R^{dN}$, where $x_N(\tau) = \sum_{k=0}^{N-1} \xi_k c_k(\tau)$ is the continuous loop with coordinates $\{\xi_k\}_{k=0}^{N-1}$.
\end{lemma}
\begin{proof}
The Hessian matrix of $V_{N,D}^c(\xi)$ is given by
\begin{equation*}
	\nabla_{\xi_k\xi_l}^2 \mathcal V_{N,D}^c(\xi) = \beta_D \sum_{j=0}^{D-1} \nabla^2 V^c(x_N(j\beta_D))
	c_k(j\beta_D) c_l(j\beta_D) \in \mathbb R^{d\times d},~~~~k,l=0,1,\cdots,N-1.
\end{equation*}
To prove that $V_{N,D}^c(\xi)$ is convex, we only need to show for any constants $\{q_k\}_{k=0}^{N-1}$,
\begin{equation*}
	\mathcal S(q,q) := \sum_{k,j=0}^{N-1} q_k\cdot \nabla_{\xi_k\xi_l}^2 \mathcal V_{N,D}^c(\xi) \cdot q_l \Ge 0.
\end{equation*}
By direct calculation, we have
\begin{equation*}
	\mathcal S(q,q) = \beta_D \sum_{j=0}^{D-1} \bigg[\bigg(\sum_{k=0}^{N-1} q_k c_k(j\beta_D)\bigg) \cdot
\nabla^2 V^c(x_N(j\beta_D))  \cdot \bigg(\sum_{k=0}^{N-1} q_l c_l(j\beta_D)\bigg)\bigg].
\end{equation*}
Introduce the constants $v_j = \sum_{k=0}^{N-1} q_k c_k(j\beta_D)$ in $\mathbb R^d$ for $j = 0,1,\cdots,D-1$, then we can write
\begin{equation*}
	\mathcal S(q,q) = \beta_D \sum_{j=0}^{D-1} v_j \cdot \nabla^2 V^c(x_N(j\beta_D)) \cdot v_j.
\end{equation*}
Since each $\nabla^2 V^c(x_N(j\beta_D)) \in \mathbb R^{d\times d}$ is positive semidefinite, we conclude $\mathcal S(q,q) \Ge 0$.
\end{proof}

\begin{lemma}
\label{appendix: 2}
Let positive integers $N,D$ satisfy $N\Le D$, and
the potential $V^a(q)$ satisfies $-M_2 I_d\Prec\nabla^2 V^a(q) \Prec M_2I_d$ in $\mathbb R^d$. For any mode coordinates $\{\xi_k\}_{k=0}^{N-1}$, define the potential function
\begin{equation*}
	\mathcal V_{N,D}^a(\xi) = \beta_D \sum_{j=0}^{D-1} V^a(x_N(j\beta_D)),~~~~\xi\in\mathbb R^{dN}
\end{equation*}
and the matrix $\Sigma \in \mathbb R^{dN}$ by
\begin{equation*}
	\Sigma_{kl} = \frac1{\sqrt{(\omega_k^2+a)(\omega_l^2+a)}} \nabla_{\xi_k\xi_l}^2 \mathcal V_{N,D}^a(\xi) \in \mathbb R^{d\times d},~~~~
	k,l=0,1,\cdots,N-1,
\end{equation*}
then $\Sigma$ satisfies $-\frac{M_2}{a} I_{dN} \Prec \Sigma \Prec \frac{M_2}{a} I_{dN}$.
\end{lemma}
\begin{proof}
By direct calculation, the Hessian of $\mathcal V_{N,D}^a(\xi)$ is given by
\begin{equation*}
\nabla_{\xi_k\xi_l}^2 \mathcal V_{N,D}^a(\xi) = \beta_D \sum_{j=0}^{D-1} \nabla^2 V^a(x_N(j\beta_D)) c_k(j\beta_D) c_l(j\beta_D)
\in \mathbb R^{d\times d},
\end{equation*}
then for any constants $\{q_k\}_{k=0}^{N-1}$ in $\mathbb R^d$, we have
\begin{equation*}
	\sum_{k,l=0}^{N-1} q_k\cdot \Sigma_{kl}\cdot q_l = \beta_D \sum_{j=0}^{D-1}
	\bigg(\sum_{k=0}^{N-1} \frac{q_k}{\sqrt{\omega_k^2+a}} c_k(j\beta_D) \bigg) \cdot
	\nabla^2 V^a(x_N(j\beta_D)) \cdot \bigg(\sum_{l=0}^{N-1} \frac{q_l}{\sqrt{\omega_l^2+a}} c_l(j\beta_D) \bigg).
\end{equation*}
Next, it is convenient to introduce the vectors
\begin{equation*}
	v_j = \sum_{k=0}^{N-1} \frac{q_k}{\sqrt{\omega_k^2+a}} c_k(j\beta_D) \in \mathbb R^d,~~~~ j = 0,1,\cdots,D-1,
\end{equation*}
then we obtain the relation
\begin{equation*}
	\sum_{k,l=0}^{N-1} q_k\cdot \Sigma_{kl}\cdot q_l = \beta_D \sum_{j=0}^{D-1}
	v_j \cdot \nabla^2 V^a(x_N(j\beta_D)) \cdot v_j \Longrightarrow
	\bigg|\sum_{k,l=0}^{N-1} q_k\cdot \Sigma_{kl}\cdot q_l\bigg| \Le \beta_D M_2 \sum_{j=0}^{D-1} |v_j|^2.
\end{equation*}
To estimate the summation $\sum_{j=0}^{D-1} |v_j|^2$, we write
\begin{equation*}
\sum_{j=0}^{D-1} |v_j|^2 = \sum_{k,l=0}^{N-1} \frac{q_k\cdot q_l}{\sqrt{(\omega_k^2+a)(\omega_l^2+a)}} \sum_{j=0}^{D-1} c_k(j\beta_D) c_l(j\beta_D).
\end{equation*}
Since the integers $k,l\Le D$, we have the discrete orthogonal condition
\begin{equation}
	\beta_D\sum_{j=0}^{D-1} c_k(j\beta_D) c_l(j\beta_D) = \left\{
	\begin{aligned}
	 & 0~\mathrm{or}~1, && \mbox{if $k=j$,} \\
	 & 0, && \mbox{if $k\neq j$.}
	\end{aligned}
	\right.
	\label{discrete normalize}
\end{equation}
Note that LHS of \eqref{discrete normalize} can be 0 when $k=j=D$ is an even integer.
Finally, we obtain
\begin{equation}
	\beta_D\sum_{j=0}^{D-1} |v_j|^2 \Le \frac1{a} \sum_{k=0}^{N-1} |q_k|^2 \Longrightarrow
	\bigg|\sum_{k,l=0}^{N-1} q_k\cdot \Sigma_{kl}\cdot q_l\bigg| \Le \frac1a\sum_{k=0}^{N-1} |q_k|^2.
	\label{discrete relation}
\end{equation}
Since \eqref{discrete relation} holds true for any constants $\{q_k\}_{k=0}^{N-1}$ in $\mathbb R^d$, we have $-\frac{M_2}{a} I_{dN} \Prec \Sigma \Prec \frac{M_2}{a} I_{dN}$.
\end{proof}
\section{Normal mode coordinates in path integral representation}
\label{appendix: normal}
We employ a slightly different approach from \cite{simple} to derive the normal mode coordinates. For simplicity, we assume the number of modes $N$ be an odd integer.
Let $\{\xi_k\}_{k=0}^{N-1}$ the mode coordinates, and
define the continuous loop by
\begin{equation*}
	x_N(\tau) = \sum_{k=0}^{N-1} \xi_k c_k(\tau),~~~~
	\tau \in [0,\beta],
\end{equation*}
where $\{c_k(\cdot)\}_{k=0}^{N-1}$ are the Fourier basis functions defined in Section~\ref{section: Matsubara}. Suppose the discretization size in $[0,\beta]$ is also $N$, then
the $N$ bead positions of the loop are
\begin{equation}
	x_j = x_N(j\beta_N) = \sum_{k=0}^{N-1} c_k(j\beta_N),~~~~j=0,1,\cdots,N-1.
	\label{normal transform}
\end{equation}
From \eqref{discrete normalize}, the grid values of $c_k(\cdot)$ satisfy the normalization condition
\begin{equation*}
	\beta_N \sum_{j=0}^{N-1} c_k(j\beta_N)
	c_l(j\beta_N)= \delta_{k,l},~~~~
	k,l=0,1,\cdots,N-1,
\end{equation*}
where $\delta_{k,l}$ is the Kronecker delta.
Then we have the equality
\begin{equation}
	\sum_{j=0}^{N-1} |x_j|^2 = \sum_{k=0}^{N-1} |\xi_k|^2 \sum_{j=0}^{N-1} c_k^2(j\beta_N) = \frac1{\beta_N} \sum_{k=0}^{N-1} |\xi_k|^2.
	\label{normal 1}
\end{equation}
The same procedure can be used to compute the spring potential
\begin{equation*}
	\sum_{j=0}^{N-1} |x_j - x_{j+1}|^2 =
	\sum_{j=0}^{N-1} x_j\cdot(2x_j - x_{j-1} - x_{j+1}) =
	{4}\sum_{k=0}^{N-1}
	\sin^2\bigg(\frac{\lceil\frac k2\rceil\pi}{N}\bigg)|\xi_k|^2.
\end{equation*}
By defining the normal mode frequencies
\begin{equation*}
	\omega_{k,N} =
	\frac{2}{\beta_N}  \sin\bigg(\frac{\lceil\frac k2\rceil\pi}{N}\bigg),
	~~~~ k =0,1,\cdots,N-1,
\end{equation*}
we can conveniently write the spring potential as
\begin{equation}
	\frac1{\beta_N^2} \sum_{j=0}^{N-1} |x_j - x_{j+1}|^2 =
	\sum_{k=0}^{N-1} \omega_{k,N}^2 |\xi_k|^2.
	\label{normal 2}
\end{equation}

Recall that in the standard PIMD with $N$ beads, the energy of the ring polymer is
\begin{equation*}
	\mathcal E^{\std}(x) =
	\frac1{2\beta_N^2} \sum_{j=0}^{N-1} |x_j - x_{j+1}|^2 +
	\beta_N \sum_{j=0}^{N-1} V(x_j),
\end{equation*}
then we employ the equalities \eqref{normal 1} and \eqref{normal 2} to rewrite $\mathcal E^{\std}(x)$ as
\begin{align*}
\mathcal E^{\std}(x) & =
	\frac1{2\beta_N^2} \sum_{j=0}^{N-1} |x_j - x_{j+1}|^2 +
	\frac{a\beta_N}2 \sum_{j=0}^{N-1} |x_j|^2 +
	\beta_N \sum_{j=0}^{N-1} V^a(x_j) \\
	& = \frac12\sum_{j=0}^{N-1} (\omega_{k,N}^2+a) |\xi_k|^2
	+ \beta_N \sum_{j=0}^{N-1} V^a
	\bigg(\sum_{k=0}^{N-1} \xi_k c_k(j\beta_N)\bigg).
\end{align*}
Therefore, the classical Boltzmann distribution in the standard PIMD is given by
\begin{equation*}
	\exp\bigg\{
	-\frac12\sum_{k=0}^{N-1}
	(\omega_{k,N}^2+a) |\xi_k|^2 - \beta_N\sum_{j=0}^{N-1} V^a
	\bigg(\sum_{k=0}^{N-1} \xi_k c_k(j\beta_N)\bigg)
	\bigg\}.
\end{equation*}
\section{Generalized $\Gamma$ calculus}
\label{appendix: review gamma}
We review the generalized $\Gamma$ calculus developed in \cite{m1,m2}, which deals with the ergodicity
of the Markov processes with degenerate diffusions, for example, the underdamped Langevin dynamics. The generalized $\Gamma$ calculus
is based on the Bakry--\'Emery theory \cite{bakry}.

Let $\{X_t\}_{t\Ge0}$ be a reversible Markov process in $\mathbb R^d$, and $(P_t)_{t\Ge0}$ be the Markov semigroup. Let $L$ be the infinitesimal generator of $\{X_t\}_{t\Ge0}$, and $\pi$ be the invariant distribution. Then $L$ is self-adjoint in $L^2(\pi)$.
In the classical Bakry--\'Emery theory, the carr\'e du champ operator $\Gamma_1(f,g)$ and the iterated operator $\Gamma_2(f,g)$ are defined by
\begin{align*}
	\Gamma_1(f,g) & = \frac12(L(fg) - gLf - fLg), \\
	\Gamma_2(f,g) & = \frac12(L\Gamma_1(f,g)
	- \Gamma_1(f,Lg) - \Gamma_1(g,Lf)).
\end{align*}
The curvature-dimension condition $CD(\rho,\infty)$ is known as the function inequality
\begin{equation*}
	\Gamma_2(f,f) \Ge \rho \Gamma_1(f,f),~~~~
	\mbox{for any smooth function $f$}.
\end{equation*}
For a given positive smooth function $f$ in $\mathbb R^d$, define the relative entropy of $f$ with respect to the invariant distribution $\pi$ by
\begin{equation*}
	\Ent_\pi(f) = \int_{\mathbb R^d}
	f\log f\d\pi - \int_{\mathbb R^d} f\d\pi
	\log \int_{\mathbb R^d} f\d\pi.
\end{equation*}
If $\rho>0$, then $CD(\rho,\infty)$ implies the log-Sobolev inequality (see Equation (5.7.1) of \cite{bakry})
\begin{equation}
	\Ent_{\pi}(f) \Le \frac1{2\rho}
	\int_{\mathbb R^d}
	\frac{\Gamma_1(f,f)}{f}\d\pi,~~~~
	\mbox{for any positive smooth function $f$}.
	\label{log Sobolev}
\end{equation}
and thus the exponential decay of the relative entropy (see Theorem 5.2.1 of \cite{bakry})
\begin{equation}
	\Ent_\pi(P_tf) \Le e^{-2\rho t} \Ent_\pi(f),~~~~
	\forall t\Ge0.
	\label{exponential decay}
\end{equation}
Inspired from the operators $\Gamma_1$ and $\Gamma_2$, we define the generalized $\Gamma$ operator as follows.
\begin{definition}
\label{definition: gamma}
Suppose $f$ is a smooth function, and $\Phi(f)$ is a function of $f$ and the derivatives of $f$. For a stochastic process $\{X_t\}_{t\Ge0}$ with generator $L$, define the generalized $\Gamma$ operator by
\begin{equation*}
	\Gamma_{\Phi}(f) = \frac12(L\Phi(f)-\d\Phi(f)\cdot Lf),
\end{equation*}
where $\d \Phi(f)\cdot g$ for two smooth functions $f,g$ are given by
\begin{equation*}
	\d \Phi(f)\cdot g = \lim_{s\rightarrow 0}
	\frac{\Phi(f+sg) - \Phi(f)}s.
\end{equation*}
\end{definition}
\noindent
The expression of $\Gamma_\Phi(f)$ can be obtained via the following result (Lemma 5 of \cite{m2}).
\begin{lemma}
\label{lemma: calculation of gamma}
Suppose $C_1,C_2$ are two linear operators and $\Phi(f) = C_1f\cdot C_2f$, then
\begin{equation*}
	\Gamma_{\Phi}(f) = \Gamma_1(C_1f, C_2f) +
	\frac12 C_1f \cdot [L,C_2]f +
	\frac12 [L,C_1]f\cdot  C_2f,
\end{equation*}
where $\Gamma_1(\cdot,\cdot)$ is the classical carr\'e du champ operator.
\end{lemma}

In the following we assume the stochastic process $\{X_t\}_{t\Ge0}$ in $\mathbb R^d$ is solved by
\begin{equation*}
	\d X_t = b(X_t)\d t + \sigma\d B_t,~~~~
	t\Ge0,
\end{equation*}
where $b:\mathbb R^d\rightarrow\mathbb R^d$ is the drift force, $\sigma\in\mathbb R^{d\times m}$ is a constant matrix, and $\{B_t\}_{t\Ge0}$ is the standard Brownian motion in $\mathbb R^m$. Then the generator of $\{X_t\}_{t\Ge0}$ is given by
\begin{equation}
	Lf(x) = b(x)\cdot f(x) + \nabla\cdot(D\nabla f),
	\label{generator simple}
\end{equation}
where $D = \frac12 \sigma\sigma^\T \in \mathbb R^{d\times d}$ is the constant diffusion matrix.
For the generator $L$ given in \eqref{generator simple}, we calculate the $\Gamma$ operator $\Gamma_\Phi(f)$ with some classical choices of $\Phi$.
\begin{example}
\label{example: 1}
If $\Phi(f) = |f|^2$, then we can take $C_1 = C_2 = 1$ in Lemma \ref{lemma: calculation of gamma} and obtain
\begin{equation*}
	\Gamma_\Phi(f) = \Gamma_1(f,f) = \frac12\big[
	\nabla(D\nabla(f^2)) - 2f\nabla(D\nabla f)
	\big]
	= (\nabla f)^\T D \nabla f.
\end{equation*}
In particular, since $D\in\mathbb R^{d\times d}$ is positive semidefinite, we always have $\Gamma_1(f) \Ge 0$.
\end{example}
\begin{example}
\label{example: 2}
If $\Phi(f) = f\log f$, then by direct calculation we have
\begin{equation*}
	\Gamma_\Phi(f) = \frac12 \Big[
	\nabla\cdot \big(D(\log f+1)\nabla f\big) -
	(\log f+1) Lf
	\Big]
	= \frac{(\nabla f)^\T D\nabla f}{2f}.
\end{equation*}
\end{example}
\begin{example}
\label{example: 4}
If $\Phi(f) = |Cf|^2/f$ for some linear operator $C$, then
\begin{equation}
	\Gamma_{\Phi}(f) \Ge \frac{Cf\cdot [L,Cf]}{f}.
	\label{Gamma estimate 2}
\end{equation}
The proof below is given in Lemma 7 of \cite{m2}.
\begin{proof}
It is easy to verify for any smooth functions $f,g$, there holds
\begin{equation*}
	L(fg) = gLf + fLg + 2\Gamma_1(f,g).
\end{equation*}
By replacing $f\rightarrow|Cf|^2$ and $g\rightarrow 1/f$, we have
\begin{align}
	L \bigg( \frac{|Cf|^2}f \bigg) & =
	\frac1f L(|Cf|^2) + |Cf|^2
	L\bigg(\frac1f\bigg) + 2\Gamma_1
	\bigg(|Cf|^2,\frac1f\bigg) \notag \\
	& = \frac1f L(|Cf|^2) + |Cf|^2
	\bigg(-\frac{Lf}{f^2} + \frac2{f^3}\Gamma_1(f)\bigg) +
	\frac{4Cf\cdot \Gamma_1(Cf,f)}{f^2}.
	\label{Gamma 1}
\end{align}
Note that
\begin{equation}
	\d\bigg(\frac{|Cf|^2}f\bigg) \cdot
	|Cf|^2 = \frac{\d(|Cf|^2)\cdot Lf}{f^2}
	-|Cf|^2 \frac{Lf}{f^2}.
	\label{Gamma 2}
\end{equation}
Hence from \eqref{Gamma 1}\eqref{Gamma 2} and the definition of the generalized $\Gamma$ operator in \eqref{definition: gamma}, we have
\begin{equation}
	\Gamma_{\Phi}(f) \Ge
	\frac{\Gamma_{|C\cdot|^2}(f)}{f} +
	\frac1{f^3} |Cf|^2 \Gamma_1(f) +
	\frac{2Cf\cdot \Gamma_1(Cf,f)}{f^2},
	\label{Gamma 3}
\end{equation}
where $\Gamma_{|C\cdot|^2}$ is the generalized
$\Gamma$ operator induced by the function $|Cf|^2$.
Since the matrix $D\in\mathbb R^{d\times d}$ is positive semidefinite, we have the Cauchy-Swarchz inequality
\begin{equation}
	|\Gamma_1(f,Cf)|^2 \Le \Gamma_1(f,f) \Gamma_1(Cf,Cf),~~~~
	\mbox{for any smooth function $f$}.
	\label{Gamma 4}
\end{equation}
From \eqref{Gamma 3} and \eqref{Gamma 4} we obtain
\begin{align*}
	\Gamma_{\Phi}(f) & \Ge
	\frac{\Gamma_1(Cf,Cf) + Cf\cdot [L,C]f}{f} +
	\frac{|Cf|^2 \Gamma_1(f,f)}{f^3} \\
	& ~~~~ - 2 \sqrt{\frac{|Cf|^2\Gamma_1(f,f)}{f^3}}
	\sqrt{\frac{\Gamma_1(Cf,Cf)}{f}} \Ge \frac{Cf\cdot [L,C]f}{f}.
\end{align*}
Hence we obtain the desired result.
\end{proof}
\end{example}

Now we establish the curvature-dimension conditions for the generalized $\Gamma$ operators. The following result relates the time derivative of $\Phi(f)$ with $\Gamma_{\Phi}(f)$.
\begin{lemma}
\label{lemma: time derivative}
Given the constant $t>0$, for any $s\in[0,t]$, we have the equality
\begin{equation*}
	\frac{\d}{\d s} \big[
	P_s \Phi(P_{t-s}f)(x) \big] = 2P_s
	\Gamma_{\Phi}(P_{t-s}f)(x).
\end{equation*}
As a consequence, for any $t\Ge0$,
\begin{equation*}
	\frac{\d}{\d t}
	\int_{\mathbb R^d}
	\Phi(P_tf)\d\pi
	= -2\int_{\mathbb R^d}
	\Gamma_\Phi(P_tf)\d\pi.
\end{equation*}
\end{lemma}
\begin{proof}
Using the chain rule, we have
\begin{align*}
	\frac{\d}{\d s}\big[
	P_s \Phi(P_{t-s} f)
	\big] & = LP_s \Phi(P_{t-s} f) + P_s
	\frac{\d}{\d s} \big[
	\Phi(P_{t-s}f)
	\big] \\
	& = LP_s \Phi(P_{t-s}f) + P_s
	\lim_{r\rightarrow 0}
	\frac{\Phi(P_{t-s-r}f) - \Phi(P_{t-s}f)}{r} \\
	& = LP_s \Phi(P_{t-s}f) - P_s \d \Phi(P_{t-s}f)
	\cdot L P_{t-s}f \\
	& = P_s \big(
	L \Phi_{t-s}f - \d \Phi(P_{t-s}f) \cdot L P_{t-s}f
	\big) \\
	& = 2P_s \Gamma_{\Phi}(P_{t-s}f).
\end{align*}
Integrating the equality over the distribution $\pi$, we obtain
\begin{equation}
	\frac{\d}{\d s}\int_{\mathbb R^d}
	P_s\Phi(P_{t-s}f) \d\pi =
	2\int_{\mathbb R^d} \Gamma_{\Phi}(P_{t-s}f)\d\pi.
	\label{middle 6}
\end{equation}
Since $\pi$ is the invariant distribution, replacing $t-s$ by $s$ in \eqref{middle 6}, we obtain
\begin{equation*}
	\frac{\d}{\d s} \int_{\mathbb R^d}
	\Phi(P_tf)\d\pi = -2\int_{\mathbb R^d}
	\Gamma_{\Phi}(P_t f)\d\pi,
\end{equation*}
which completes the proof.
\end{proof}
\noindent
By choosing $\Phi(f) = f\log f$, Lemma \ref{lemma: time derivative} implies
\begin{equation}
	\frac{\d}{\d t}\Ent_\pi(P_t f) =
	\frac{\d}{\d t} \int_{\mathbb R^d}
	\Phi(P_tf) \d t = -\int_{\mathbb R^d}
	\frac{(\nabla f)^\T D\nabla f}{f}\d\pi = -\int_{\mathbb R^d} \frac{\Gamma_1(f,f)}{f}\d\pi.
	\label{middle 8}
\end{equation}
If we have the log-Sobolev inequality \eqref{log Sobolev}, from \eqref{middle 8} we have
\begin{equation*}
	\frac{\d}{\d t}\Ent_{\pi}(P_tf) \Le -2\rho \Ent_{\pi}(f),
\end{equation*}
which implies $\Ent_{\pi}(P_tf) \Le e^{-2\rho t}\Ent_\pi(f)$, and we recover the result in \eqref{exponential decay}.

Now we state the main theorem, which provides the generalized curvature-dimension condition for degenerate diffusion processes.
\begin{theorem}
\label{theorem: generalized}
Let $\{X_t\}_{t\Ge0}$ be an ergodic stochastic process with the invariant distribution $\pi$. If for two functions $\Phi_1(f)$ and $\Phi_2(f)$, there hold the functional inequalities
\begin{equation}
	0 \Le \int_{\mathbb R^d} \Phi_1(f)\d\pi -
	\Phi_1\bigg(
	\int_{\mathbb R^d} f\d\pi
	\bigg) \Le c
	\int_{\mathbb R^d} \Phi_2(f)\d\pi,
	\label{middle 9}
\end{equation}
\begin{equation}
	\Gamma_{\Phi_2}(f) \Ge \rho \Phi_2(f) - m
	\Gamma_{\Phi_1}(f),
	\label{generalized CD}
\end{equation}
for some constants $c,\rho,m>0$, then by defining the entropy-like quantity
\begin{equation*}
	W_\pi(f) = m\bigg(
	\int_{\mathbb R^d}
	\Phi_1(f)\d\pi -
	\Phi_1\bigg(
	\int_{\mathbb R^d} f\d\pi\bigg)
	\bigg) + \int_{\mathbb R^d} \Phi_2(f)\d\pi,
\end{equation*}
we have the exponential decay
\begin{equation*}
	W_\pi(P_tf) \Le \exp\bigg(
	-\frac{2\rho t}{1+m c}
	\bigg) W_\pi(f),~~~~ \forall t\Ge 0.
\end{equation*}
\end{theorem}
\noindent
The proof below is given in Lemma 3 of \cite{m2}.
\begin{proof}
Using Lemma \ref{lemma: time derivative} and \eqref{generalized CD}, we have
\begin{align}
	\frac{\d}{\d t} W_\pi(P_tf) & = \frac{\d}{\d t}\bigg[
	m \int_{\mathbb R^d} \Phi_1(P_tf)\d\pi +
	 \int_{\mathbb R^d} \Phi_2(P_tf)\d\pi\bigg] \notag \\
	& = -2\int_{\mathbb R^d} \big(
	m \Gamma_{\Phi_1} + \Gamma_{\Phi_2}
	\big)(P_tf)\d\pi \Le -2\rho \int_{\mathbb R^d} \Phi_2(P_tf)\d\pi.
	\label{middle 10}
\end{align}
Using \eqref{middle 9} and the definition of $W(f)$, we have
\begin{equation*}
	W_\pi(f) = m\bigg(
		\int_{\mathbb R^d}
		\Phi_1(f)\d\pi -
		\Phi_1\bigg(
		\int_{\mathbb R^d} f\d\pi\bigg)
		\bigg) + \int_{\mathbb R^d} \Phi_2(f)\d\pi \Le
	(1+m c) \int_{\mathbb R^d} \Phi_2(f)\d\pi.
\end{equation*}
Hence \eqref{middle 10} implies
\begin{equation*}
	\frac{\d}{\d t} W_\pi(P_tf) \Le -\frac{2\rho}{1+m c} W_\pi(P_tf),~~~~\forall t\Ge0,
\end{equation*}
yielding the desired result.
\end{proof}

\end{appendices}

\bibliography{reference}

\begin{thebibliography}{10}

\bibitem{q1}
Kerson Huang.
\newblock {Statistical mechanics}.
\newblock {\em John Wiley \& Sons}, 2008.

\bibitem{q2}
Donald~A. McQuarrie.
\newblock {Statistical mechanics}.
\newblock {\em Sterling Publishing Company}, 2000.

\bibitem{q3}
Peter~W. Atkins, Julio De~Paula, and James Keeler.
\newblock {\em Atkins' physical chemistry}.
\newblock Oxford university press, 2023.

\bibitem{q4}
Neil~W. Ashcroft and N.~David Mermin.
\newblock {Solid state physics}.
\newblock {\em Cengage Learning}, 2022.

\bibitem{q5}
Subir Sachdev.
\newblock {Quantum phase transitions}.
\newblock {\em Physics world}, 12(4):33, 1999.

\bibitem{s1}
Steven~A. Orszag.
\newblock {Comparison of pseudospectral and spectral approximation}.
\newblock {\em Studies in Applied Mathematics}, 51(3):253--259, 1972.

\bibitem{s2}
Jie Shen, Tao Tang, and Li-Lian Wang.
\newblock {Spectral methods: algorithms, analysis and applications}.
\newblock {\em Springer}, 41, 2011.

\bibitem{feynman}
Richard~P. Feynman, Albert~R. Hibbs, and Daniel~F. Styer.
\newblock {Quantum mechanics and path integrals}.
\newblock {\em Courier Corporation}, 2010.

\bibitem{pimd_6}
William~H. Miller.
\newblock {Path integral representation of the reaction rate constant in
  quantum mechanical transition state theory}.
\newblock {\em The Journal of Chemical Physics}, 63(3):1166--1172, 1975.

\bibitem{pimd_9}
Ian~R. Craig and David~E. Manolopoulos.
\newblock {Chemical reaction rates from ring polymer molecular dynamics}.
\newblock {\em The Journal of chemical physics}, 122(8):084106, 2005.

\bibitem{pimd_10}
Xuecheng Tao, Philip Shushkov, and Thomas~F. Miller.
\newblock {Microcanonical rates from ring-polymer molecular dynamics:
  Direct-shooting, stationary-phase, and maximum-entropy approaches}.
\newblock {\em The Journal of Chemical Physics}, 152(12), 2020.

\bibitem{pimd_7}
Gregory~A. Voth.
\newblock {Feynman path integral formulation of quantum mechanical
  transition-state theory}.
\newblock {\em The Journal of Physical Chemistry}, 97(32):8365--8377, 1993.

\bibitem{ts_1}
Edit M{\'a}tyus, David~J. Wales, and Stuart~C. Althorpe.
\newblock {Quantum tunneling splittings from path-integral molecular dynamics}.
\newblock {\em The Journal of chemical physics}, 144(11):114108, 2016.

\bibitem{ts_2}
Christophe~L. Vaillant, David~J. Wales, and Stuart~C. Althorpe.
\newblock {Tunneling splittings from path-integral molecular dynamics using a
  Langevin thermostat}.
\newblock {\em The Journal of chemical physics}, 148(23):234102, 2018.

\bibitem{rpmd_1}
Scott Habershon, David~E. Manolopoulos, Thomas~E. Markland, and Thomas~F.
  Miller~III.
\newblock {Ring-polymer molecular dynamics: Quantum effects in chemical
  dynamics from classical trajectories in an extended phase space}.
\newblock {\em Annual review of physical chemistry}, 64:387--413, 2013.

\bibitem{rpmd_2}
Ian~R. Craig and David~E. Manolopoulos.
\newblock {Quantum statistics and classical mechanics: Real time correlation
  functions from ring polymer molecular dynamics}.
\newblock {\em The Journal of chemical physics}, 121(8):3368--3373, 2004.

\bibitem{cmd_1}
Seogjoo Jang and Gregory~A. Voth.
\newblock {A derivation of centroid molecular dynamics and other approximate
  time evolution methods for path integral centroid variables}.
\newblock {\em The Journal of chemical physics}, 111(6):2371--2384, 1999.

\bibitem{cmd_2}
Jianshu Cao and Gregory~A. Voth.
\newblock {The formulation of quantum statistical mechanics based on the
  Feynman path centroid density. III. Phase space formalism and analysis of
  centroid molecular dynamics}.
\newblock {\em The Journal of chemical physics}, 101(7):6157--6167, 1994.

\bibitem{Matsubara_0}
Michael~J. Willatt.
\newblock {\em {Matsubara dynamics and its practical implementation}}.
\newblock PhD thesis, Cambridge University, 2017.

\bibitem{Matsubara_1}
Timothy J.~H. Hele, Michael~J. Willatt, Andrea Muolo, and Stuart~C. Althorpe.
\newblock {Boltzmann-conserving classical dynamics in quantum time-correlation
  functions: ``Matsubara dynamics"}.
\newblock {\em The Journal of Chemical Physics}, 142(13):134103, 2015.

\bibitem{pimd_3}
Jianshu Cao and Glenn~J. Martyna.
\newblock {Adiabatic path integral molecular dynamics methods. II. Algorithms}.
\newblock {\em The Journal of chemical physics}, 104(5):2028--2035, 1996.

\bibitem{pimd_4}
Mark~E. Tuckerman, Dominik Marx, Michael~L. Klein, and Michele Parrinello.
\newblock {Efficient and general algorithms for path integral Car--Parrinello
  molecular dynamics}.
\newblock {\em The Journal of chemical physics}, 104(14):5579--5588, 1996.

\bibitem{pimd_5}
Thomas~E. Markland and David~E. Manolopoulos.
\newblock {An efficient ring polymer contraction scheme for imaginary time path
  integral simulations}.
\newblock {\em The Journal of chemical physics}, 129(2):024105, 2008.

\bibitem{pimd_8}
Bruce~J. Berne and D.~Thirumalai.
\newblock {On the simulation of quantum systems: path integral methods}.
\newblock {\em Annual Review of Physical Chemistry}, 37(1):401--424, 1986.

\bibitem{multi_1}
Jianfeng Lu and Zhennan Zhou.
\newblock {Path integral molecular dynamics with surface hopping for thermal
  equilibrium sampling of nonadiabatic systems}.
\newblock {\em The Journal of Chemical Physics}, 146(15), 2017.

\bibitem{multi_2}
Xinzijian Liu and Jian Liu.
\newblock {Path integral molecular dynamics for exact quantum statistics of
  multi-electronic-state systems}.
\newblock {\em The Journal of Chemical Physics}, 148(10), 2018.

\bibitem{pmmLang}
Jianfeng Lu, Yulong Lu, and Zhennan Zhou.
\newblock {Continuum limit and preconditioned Langevin sampling of the path
  integral molecular dynamics}.
\newblock {\em Journal of Computational Physics}, 423:109788, 2020.

\bibitem{Cayley_2}
Roman Korol, Nawaf Bou-Rabee, and Thomas~F. Miller.
\newblock {Cayley modification for strongly stable path-integral and
  ring-polymer molecular dynamics}.
\newblock {\em The Journal of chemical physics}, 151(12), 2019.

\bibitem{Matsubara}
Rob~D. Coalson.
\newblock {On the connection between Fourier coefficient and Discretized
  Cartesian path integration}.
\newblock {\em The Journal of chemical physics}, 85(2):926--936, 1986.

\bibitem{Matsubara_compare}
Charusita Chakravarty, Maria~C. Gordillo, and David~M. Ceperley.
\newblock {A comparison of the efficiency of Fourier-and discrete time-path
  integral Monte Carlo}.
\newblock {\em The Journal of chemical physics}, 109(6):2123--2134, 1998.

\bibitem{bakry}
Dominique Bakry, Ivan Gentil, Michel Ledoux, et~al.
\newblock {Analysis and geometry of Markov diffusion operators}.
\newblock {\em Springer}, 103, 2014.

\bibitem{m1}
Pierre Monmarch{\'e}.
\newblock {Hypocoercivity in metastable settings and kinetic simulated
  annealing}.
\newblock {\em Probability Theory and Related Fields}, 172(3-4):1215--1248,
  2018.

\bibitem{m2}
Pierre Monmarch{\'e}.
\newblock {Generalized $\Gamma$ calculus and application to interacting
  particles on a graph}.
\newblock {\em Potential Analysis}, 50(3):439--466, 2019.

\bibitem{path}
James Glimm and Arthur Jaffe.
\newblock {Quantum physics: a functional integral point of view}.
\newblock {\em Springer Science \& Business Media}, 2012.

\bibitem{md_2}
Xuda Ye and Zhennan Zhou.
\newblock Efficient sampling of thermal averages of interacting quantum
  particle systems with random batches.
\newblock {\em The Journal of Chemical Physics}, 154(20), 2021.

\bibitem{pHMC}
Nawaf Bou-Rabee and Andreas Eberle.
\newblock {Two-scale coupling for preconditioned Hamiltonian Monte Carlo in
  infinite dimensions}.
\newblock {\em Stochastics and Partial Differential Equations: Analysis and
  Computations}, 9:207--242, 2021.

\bibitem{simple}
Jian Liu, Dezhang Li, and Xinzijian Liu.
\newblock {A simple and accurate algorithm for path integral molecular dynamics
  with the Langevin thermostat}.
\newblock {\em The Journal of chemical physics}, 145(2):024103, 2016.

\bibitem{explicit}
Yu~Cao, Jianfeng Lu, and Lihan Wang.
\newblock {On explicit $L^2$-convergence rate estimate for underdamped Langevin
  dynamics}.
\newblock {\em arXiv preprint arXiv:1908.04746}, 2019.

\bibitem{villani}
C{\'e}dric Villani.
\newblock {\em {Hypocoercivity}}.
\newblock Number 949-951. American Mathematical Soc., 2009.

\bibitem{md}
Ben Leimkuhler and Charles Matthews.
\newblock {Molecular dynamics}.
\newblock {\em Interdisciplinary applied mathematics}, 39:443, 2015.

\bibitem{md_1}
Marco Lauricella, Letizia Chiodo, Fabio Bonaccorso, Mihir Durve, Andrea
  Montessori, Adriano Tiribocchi, Alessandro Loppini, Simonetta Filippi, and
  Sauro Succi.
\newblock {Multiscale Hybrid Modeling of Proteins in Solvent: SARS-CoV2 Spike
  Protein as test case for Lattice Boltzmann--All Atom Molecular Dynamics
  Coupling}.
\newblock {\em arXiv preprint arXiv:2305.05025}, 2023.

\end{thebibliography}
\bibliographystyle{unsrt}

\end{document}